\documentclass[a4paper]{amsart}
\usepackage[english]{babel}
\usepackage[utf8]{inputenc}
\usepackage{graphicx}
\usepackage{amsmath}
\usepackage{amsfonts,amssymb,amsthm}
\usepackage{txfonts}

\usepackage[active]{srcltx}

\usepackage{url}
\newcommand{\Fock}{\mathcal{F}^{2}(\mathbb{C}^{n})}
\newcommand{\Complexn}{\mathbb{C}^{n}}
\newcommand{\BFn}{\mathcal{B}(\mathcal{F}^{2}(\mathbb{C}^{n}))}

\newcommand{\Pn}{P_{n}}

\newcommand{\Complex}{\mathbb{C}}

\newcommand{\C}{\mathcal{C}}

\newcommand{\K}{\mathcal{K}}
\newcommand{\g}{\mathrm{g}}

\newcommand{\RE}{\mathrm{Re}}
\newcommand{\IM}{\mathrm{Im}}

\newcommand{\Ln}{\mathcal{L}}
\newcommand{\Lag}{\mathcal{L}}
\newcommand{\B}{\mathcal{B}}

\newcommand{\M}{\mathcal{M}}
\newcommand{\A}{\mathcal{A}}

\newcommand{\W}{\mathcal{W}}

\newcommand{\F}{\operatorname{F}}
\newcommand{\N}{\mathbb{N}}
\newcommand{\innerp}[2]{\left\langle#1,#2\right\rangle}

\newcommand{\Real}{\mathbb{R}}
\newcommand{\Realn}{\mathbb{R}^{n}}
\newcommand{\Entero}{\mathbb{Z}}

\newcommand{\h}{\mathcal{H}}

\newcommand{\Ele}{L_{2}}

\newcommand{\Borelpcn}{\mathfrak{B}_{reg}(\Complexn)}
\newcommand{\Borelpcnm}{\mathfrak{B}_{reg,M}(\Complexn)}

\newcommand{\Borelcn}{\operatorname{Borel}(\Complexn)}

\newcommand{\Borelrn}{\operatorname{Borel}(\mathbb{R}^{n})}
\newcommand{\Borelprn}{\mathfrak{B}_{reg}(\Realn)}

\newcommand{\ppabeta}{\ppartial^{\alpha}\overline{\ppartial}{}^{\beta}}
\newcommand{\hspn}{\hspace{-2pt}}

\newcommand{\Frc}{\operatorname{G}}
\newcommand{\Bb}{\mathbf{B}}
\newcommand{\rb}{\mathbf{r}}
\newcommand{\bT}{\mathbf{T}}
\newtheorem{theorem}{Theorem}[section]
\newtheorem{corollary}[theorem]{Corollary}
\newtheorem{proposition}[theorem]{Proposition}
\newtheorem{definition}[theorem]{Definition}
\newtheorem{lemma}[theorem]{Lemma}

\theoremstyle{definition}
\newtheorem{example}{Example}
\newtheorem{remark}{Remark}

\newcommand{\ppartial}{\pmb{\partial}}

\numberwithin{equation}{section}

\begin{document}

\title[$\Lag$-invariant FC measures for derivatives and Toeplitz operators]{$\Lag$-invariant Fock-Carleson type measures \\ for derivatives of order $k$ and the \\ corresponding Toeplitz operators}

	\author[Esmeral]{Kevin Esmeral}
	\address{%
		Department of Mathematics\\Universidad de Caldas\\
		P.O. 170004\\
		Manizales,	Colombia}
	
	\email{kevin.esmeral@ucaldas.edu.co}
	\author[Rozenblum]{Grigori Rozenblum}
	\address{ Department of Mathematics\\ Chalmers University of Technology,  The University of Gothenburg, Sweden, \\ and Faculty of Physics St.Petersburg State University, Russia}
	\email{grigori@chalmers.se}
	\author[Vasilevski]{Nikolai Vasilevski}
	\address{Department of Mathematics\\CINVESTAV-IPN \\M\'exico, D.F, 07360, M\'EXICO}
	\email{nvasilev@math.cinvestav.mx}
	\thanks{The first named author wishes to thank the  Universidad de Caldas for financial support and hospitality.
	\newline
The second named author is grateful for support by grant RFBR  No 17-01-00668 and to Mittag-Leffler Institute for support and productive environment.
\newline
	The third named author has been partially supported by CONACYT Project 102800, M\'exico.}
	\subjclass{47A75 (primary), 58J50 (secondary)}
	\keywords{Fock space, Toeplitz operators}
	
	\date{\today}
	

\dedicatory{To Nina Uraltseva, with best regards on her Jubilee}
	\begin{abstract}
		\noindent
		Our purpose is to characterize the so-called  horizontal Fock-Carleson type measures for derivatives of order $k$ (we write it $k$-hFC for short) for the Fock space as well as the Toeplitz operators generated by sesquilinear forms given by them. We introduce  real coderivatives of  $k$-hFC type measures and show that the C*-algebra generated by  Toeplitz operators with the corresponding class of symbols is commutative and isometrically isomorphic to certain C*-subalgebra of $L_{\infty}(\mathbb{R}^{n})$. The above results are extended  to measures that are invariant under translations along Lagrangian planes.
	\end{abstract}

	\maketitle

	\section{Introduction}

In the study of Toeplitz operators in the Bergman type spaces, certain algebraic and analytic questions arise. The algebraic ones deal with describing  algebras generated by Toeplitz operators with particular  classes of symbols. One of the most important cases here concerns the algebra in question being commutative, and the Toeplitz operators generating this algebra admitting a simultaneous diagonalization: this means that there exists a unitary operator (specific for each particular case) that transforms all Toeplitz operators in this algebra into  operators of multiplication by some  functions (we call them \emph{spectral functions}). Of course, such a diagonalization immediately reveals all the main properties of corresponding Toeplitz operators.

The analytic questions are related with criteria for boundedness and compactness  of operators with symbols in a given class, passing over from regular to more and more singular objects serving as symbols.

The present paper, dealing with both types of questions,  has two objectives. Firstly, we extend to the multidimensional case   the construction of Toeplitz operators with strongly singular symbols, and secondly, which is our main goal, we unhide among those strongly singular symbols the classes that generate via corresponding Toeplitz operators commutative algebras. Moreover we present  explicit formulas for the spectral functions for these Toeplitz operators.

	We denote by $\Fock$ the Fock
	(also known as the Segal--Bargmann,
	see \cite{Bargmann,Fock,Segal}) space
	consisting of all entire functions that are square integrable
	with respect to the $2n$-dimensional  Gaussian measure
	\[
	d\g_{n}(z)=(\pi)^{-n}\,e^{-|z|^{2}} d\nu_{2n}(z),\quad z\in\Complexn,
	\]
	where $\nu_{2n}$ is the standard Lebesgue measure on $\Complexn\simeq\Real^{2n}$. It is well-known that $\Fock$ is a closed subspace of $\Ele(\Complexn,d\g_{n})$ and the orthogonal projection $\Pn$ from $L_{2}(\Complexn,d\g_{n})$ onto $\Fock$, called  the Bargmann projection,  has the  integral form
	\begin{equation}\label{proj-L2-F}
	(\Pn f)(z)=\int_{\Complexn}f(w)\overline{K_{z}(w)}d\g_{n}(w),
	\end{equation}
	where the function $K_{z}\colon\Complexn\rightarrow\Complexn$ is the \emph{reproducing kernel} at the point $z$,

\begin{equation*}
	K_{z}(w)=e^{\overline{z}\cdot w}, \quad\,w\in\Complexn.
\end{equation*}

	Given $\varphi\in\,L_{\infty}(\Complexn)$,
	the \emph{Toeplitz operator} $\bT_{\varphi}$, with defining symbol $\varphi$,
	acts on the Fock space $\Fock$  by  $\bT_{\varphi}f=\Pn(f\varphi)$.
	By  \eqref{proj-L2-F},  the Toeplitz operator $\bT_{\varphi}$ with bounded  $\varphi$ has the   representation
	\begin{equation}\label{int-rep-toeplit-u}
	(\bT_{\varphi}f)(z)=\pi^{-n}\int_{\Complexn}f(w)e^{z\cdot \overline{w}}e^{-|w|^{2}}\varphi(w)d\nu_{2n}(w),\quad z\in\Complexn.
	\end{equation}
	Isralowitz and Zhu \cite{Isra-Zhu} extended the above notion of Toeplitz operators  to the case of symbols being Borel regular measures $\mu$:
	\begin{equation}\label{Toeplitz-measure}
	 \bT_{\mu}f(z)=\pi^{-n}\int_{\Complexn}e^{z\cdot\overline{w}}f(w)e^{-|w|^{2}}d\mu(w),\quad z\in\Complexn.
	\end{equation}
	If $\mu$ is a  complex Borel regular measure  satisfying the \emph{Condition}  $(M)$, namely
	\begin{equation}\label{cond-M}
	 \sup_{z\in\Complexn}\int_{\Complexn}|K_{z}(w)|^{2}e^{-|w|^{2}}d|\mu|(w)<\infty,
	\end{equation}
	then the operator $\bT_{\mu}$ given in \eqref{Toeplitz-measure} is  well-defined on the dense subset of all finite linear combinations  of the kernel functions at different points $z$. In case when $\mu$ is absolutely continuous with respect to the standard Lebesgue measure, all the results  can be reformulated in terms of the density function, and  $\bT_{\mu}$ becomes a Toeplitz operator in the classic sense \eqref{int-rep-toeplit-u}.

	\smallskip	

A further extension of the notion of Toeplitz operators was elaborated in \cite{Vasilevski-Rozenblum}, where Rozenblum and Vasilevski introduced  Toeplitz  operators $\bT_{\F}$, defined by  bounded sesquilinear forms $\F$ in a reproducing kernel Hilbert space. This approach permits them to enrich the class of Toeplitz operators and turn into Toeplitz many operators that failed to be Toeplitz in the classical sense.
 	
\smallskip
In the present paper we introduce  horizontal (Definition \ref{hor-meas}), $\alpha$-horizontal (Definition \ref{ahor-meas}),  $\Ln$-invariant (Definition \ref{Linv-meas}), and $\alpha\Ln$-invariant (Definition \ref{aLinv-meas}) Fock-Carleson type  measures for derivatives of order $k$ in the multidimensional Fock space, and extend to Toeplitz operators defined by such symbols the main results of \cite{Esmeral-Vasilevski,Vasilevski-Rozenblum}.   Furthermore, we introduce  real coderivatives of Fock-Carleson measures  and  study  horizontal Toeplitz operators generated by all  these types of measures.
	
The paper is organized as follows.  Section \ref{preliminaries-1} contains some preliminaries: here we  fix our notation and recall some basic properties of horizontal and $\Ln$-invariant operators. In Section  \ref{Preliminaries-2} we introduce Fock-Carleson measures, Fock-Carleson measures for derivatives of order $k$, and Toeplitz operators generated by sesquilinear forms. In Section \ref{Section-Hor-FC} we introduce  horizontal Fock-Carleson type measures and show that  a Toeplitz operator with such measure-symbol is horizontal if and only if the Fock-Carleson measure is horizontal (Theorem \ref{thm-char-hfc}).
	Then we give an explicit formula for the spectral functions of such Toeplitz operators (Theorem \ref{diag-To-hFC}). In Section \ref{Section-Hor-kFC} we introduce  horizontal Toeplitz operators generated by Fock-Carleson type measures
	for derivatives of order $k$. We introduce  real coderivatives of $k$-FC type measures and show that a Toeplitz operator generated by these coderivatives is horizontal if and only if the measure is horizontal (Theorem \ref{thm-char-2kfc}). Then we  give an explicit formula for the spectral function (Theorem \ref{diag-To-khFC}).  In Section \ref{Section-Hor-akFC} we proceed  with the study of $\alpha$-horizontal type measures and show that an $\alpha$-horizontal measure is  a $k$-FC for $\Fock$ if and only if $\mu_{\alpha}$ is a Fock-Carleson type measure for $\Fock$ (Proposition \ref{a-k-hFC--a-2a-k-hFC}). We show further that the C*-algebra generated by  Toeplitz operators given by real coderivatives of order $2k$ of  $(\alpha,k)$-FC  type measures is commutative and isometrically isomorphic to a C*-subalgebra of $L_{\infty}(\Complexn)$.   In Section \ref{Section-LFC} we extend  the above results to $\Ln$-invariant measures. Finally, in Section \ref{se:common} we present th!
 e main co
 mmon properties of diagonalizable Toeplitz operators studied in the paper.
	
\section{Preliminaries} \label{preliminaries-1}

We  use the following standard notation: $z=x+iy=(z_{1},z_{2},\ldots,z_{n})\in\Complex^{n}$, where $x=(\RE\, z_{1},\ldots,\RE \,z_{n})$ and $y=(\IM \,z_{1},\ldots,\IM \,z_{n})$. For $z,w\in\Complex^{n}$ we write
	\begin{align*}
	 z\cdot w&=\sum_{k=1}^{n}z_{k}w_{k},\quad z^{2}=z\cdot z=\sum_{k=1}^{n}z_{k}^{2},\quad |z|^{2}=z\cdot\overline{z}=\sum_{k=1}^{n}|z_{k}|^{2}.
	\end{align*}

	For any multi-index $\alpha\in\Entero_{+}^{n}$ and any $z\in\Complexn$ we fix the usual notation:
	\begin{align*}
	\alpha!&=\alpha_{1}!\alpha_{2}!\cdots\alpha_{n}!,\quad |\alpha|=\sum_{k=1}\alpha_{k},\quad \alpha\leq \beta\Leftrightarrow \alpha_{j}\leq \beta_{j},\quad \alpha,\beta\in\Entero_{+}^{n}.\\
	z^{\alpha}&=z_{1}^{\alpha_{1}}\cdots z_{n}^{\alpha_{n}},\quad 	 \overline{z}^{\alpha}=\overline{z_{1}}^{\alpha_{1}}\cdots \overline{z_{n}}^{\alpha_{n}},\,
	\partial^{\alpha}f=\frac{\partial^{\alpha_{1}}}{z_{1}^{\alpha_{1}}}\cdots \frac{\partial^{\alpha_{1}}}{z_{n}^{\alpha_{n}}}f,\quad \overline{\partial}^{\alpha}f=\frac{\partial^{\alpha_{1}}}{\overline{z_{1}}^{\alpha_{1}}}\cdots \frac{\partial^{\alpha_{1}}}{\overline{z_{n}}^{\alpha_{n}}}f.
	\end{align*}
	
	\noindent
	The operator $\Bb^{*}\colon L_{2}(\Real^{n})\longrightarrow \Ele(\Complex^{n},d\g_{n})$,  given by
\begin{equation*}
	(\Bb^{*}f)(z)=\pi^{-n/4}\int_{\Real^{n}}f(x)e^{\sqrt{2}x\cdot z-\frac{x^{2}}{2}-\frac{z^{2}}{2}}\,dx,\quad z\in\Complexn,
\end{equation*}
is an isometric isomorphism from $\Ele(\Real^{n})$ onto the subspace $\Fock$ of  $\Ele(\Complex^{n},d\g_{n})$, known as the \emph{Bargmann transform}. Its adjoint operator  
$\Bb\colon \Ele(\Complex^{n},d\g_{n}) \longrightarrow\,L_{2}(\Real^{n})$ is
\begin{equation*}
	(\Bb f)(x)=\pi^{-n/4}\int_{\Complex^{n}}f(z)e^{\sqrt{2}x\cdot\overline{z}-\frac{x^{2}}{2}-\frac{\overline{z}^{2}}{2}}\,d\g_{n}(z),\quad x\in\Realn.
\end{equation*}

	\begin{example}\label{example-rnucleo}
	We calculate the image of the reproducing kernel under the inverse Bargmann transform applying \cite[Proposition 6.10]{Zhu-F} $n$ times. This yields
		\begin{align}\label{Bargmann-inversa-kernel}
		(\Bb K_{z})(x)&=\prod_{j=1}^{n}\pi^{-1/4}\int_{\Complex}e^{\overline{z_{j}}w_{j}}e^{\sqrt{2}x_{j}\overline{w_{j}}-\frac{x_{j}^{2}}{2}-\frac{\overline{w_{j}}^{2}}{2}}d\g(w_{j})=(\pi)^{-n/4}e^{\sqrt{2}x\cdot\overline{z}-\frac{x^{2}}{2}-\frac{\overline{z}^{2}}{2}}, \quad x\in\Realn.
		\end{align}
	\end{example}
	
Recall that, given $h\in\Complex^{n}$, the \emph{Weyl operator} $\W_{h}$ is defined on $\Ele(\Complexn,d\g_{n})$ as the following weighted translation
	\begin{equation*}
		\W_{h}f(z)=\displaystyle e^{z\cdot\overline{h}-\frac{|h|^{2}}{2}}f(z-h),\quad z\in\Complex^{n}.
	\end{equation*}
The Weyl operator $\W_{h}$  is  obviously unitary,  with $\W_{-h}=\W_{h}^{-1}$,  and
\begin{equation}\label{weyl-nucleok}
	\W_{h}K_{z}(w)=e^{-\overline{z}\cdot h-\frac{h^{2}}{2}}K_{z+h}(w),\quad w\in\Complex^{n}.
\end{equation}	

\medskip
We list now several statements (see \cite{Esmeral-Vasilevski}, for details), whose extensions to Toeplitz operators with singular symbols will constitute the main results of the paper.

A linear  operator  $S\in\BFn$ is said to be \emph{horizontal}  if for every $h\in\Realn$ it commutes with  $\W_{ih}$,  that is, for every $h\in\Realn$,		
\begin{equation*}
\W_{ih}S=S\W_{ih}.
\end{equation*}
Then, a function $\varphi\in\,L_{\infty}(\Complex^{n})$ is said to be  \emph{horizontal}  if for every $h\in\Realn$,
		\begin{equation*}
		\varphi(\zeta-ih)=\varphi(\zeta),\quad  \zeta\in\Complexn.
		\end{equation*}	
		
\begin{lemma}[{\cite[Lemma 3.6]{Esmeral-Vasilevski}}]\label{lema-trans-ih}
		A function  $\varphi\in\,L_{\infty}(\Complex^{n})$ is horizontal  if and only if there exists $a\in\,L_{\infty}(\Real^{n})$ such that
		$\varphi(z)=a(\RE\, z),\quad \text{a. e. $z\in\Complex^{n}$}.$
\end{lemma}
Recall that the \emph{Berezin transform} \cite{BerezinFA}  of an  operator $S\in\B(\Fock)$ is given by	
\begin{equation*}
\widetilde{S}(z)=\frac{\innerp{SK_{z}}{K_{z}}}{\innerp{K_{z}}{K_{z}}}, \quad z \in \mathbb{C}^n.
\end{equation*}
\begin{proposition}[{\cite[Theorem 3.7]{Esmeral-Vasilevski}}]\label{criterio-horizontal-operators}
		Let $S\in\B(\Fock)$. Then  the following conditions are equivalent:
		\begin{enumerate}
			\item 		$S$ is horizontal;
			\item There exists $\varphi\in\,L_{\infty}(\Realn)$ such that $\Bb S\Bb^{*}=M_{\varphi}$, where $M_{\varphi}$ denotes the multiplication operator by $\varphi$;
			\item The Berezin transform $\widetilde{S}$ is a horizontal function. i.e., $\widetilde{S}(z)=b(\RE \,z),\quad a.e.\, z\in\Complex^{n}$ for some $b\in\,L_{\infty}(\Real^{n})$. 	
		\end{enumerate}
\end{proposition}
\begin{proposition}[{\cite[Proposition 3.10]{Esmeral-Vasilevski}}]
 If $\varphi \in L_{\infty}(\mathbb{C}^n)$, then the Toeplitz operator $T_{\varphi}$ is horizontal if and only if $\varphi$ is a horizontal function.
\end{proposition}
\begin{theorem}[{\cite[Theorem 3.8]{Esmeral-Vasilevski}}]
Let $\varphi(z)= a(\RE\, z)$ be a horizontal $L_{\infty}$-function. Then the Toeplitz operator $T_{\varphi}$ is unitarily equivalent to the multiplication operator: $\Bb \bT_{\varphi}\Bb^{*}=\gamma_{a}I$, where the function $\gamma_{a}\colon\Real^{n}\longrightarrow\Complex$ is given by
	\begin{equation*}
	 \gamma_{a}(x)=\pi^{-n/2}\int_{\Real^{n}}a\left(\frac{y}{\sqrt{2}}\right)\,e^{-(x-y)^{2}}\,dy,\quad x\in\Real^{n}.
	\end{equation*}
\end{theorem}

\medskip
We recall as well that a symplectic  space $(V,\omega)$ is a vector space  $V$ equipped with a symplectic (non-degenerate skew-symmetric bilinear) form  $\omega$. The standard symplectic form $\omega_{0}$ of $\Real^{2n}= \mathbb{C}^n$ is
	\begin{equation*}
	\omega_{0}(z,w)=\mathbf{J}z\cdot w,\quad z,w\in\Real^{2n},
	\end{equation*}
	where $\mathbf{J}$ is the ``standard symplectic matrix''
	\begin{equation*}
	\mathbf{J}=\left(\begin{matrix}
	0&I_{n}\\ -I_{n}&0
	\end{matrix}
	\right),
	\end{equation*}
	and $0$ and $I_{n}$ are the $n\times n$ zero and unit matrices.		
	
	Then an $n$-dimensional linear subspace $\Ln$ of $\Real^{2n}$ is said to be a \emph{Lagrangian plane} of the symplectic space $(\Real^{2n},\omega_{0})$ if  for every  $z,w\in\Ln$,  $\omega_{0}(z,w)=0$.   We denote by $\operatorname{Lag}(2n,\Real)$ the set of all Lagrangian planes in $(\Real^{2n},\omega_{0})$.
	\begin{example}
		The simplest examples of  $\Ln\in\operatorname{Lag}(2n,\Real)$ are  both coordinates planes: $\Ln_{x}=\Real^{n}\times\{0\}$ and $\Ln_{y}=\{0\}\times\Real^{n}$, as well as the diagonal $\Delta=\{(x,x)\colon\,x\in\Real^{n}\}$.
	\end{example}
The  set of symplectic rotations   $\operatorname{U}(2n,\Real)$ forms a subgroup in the group of all linear automorphisms of $\Real^{2n}$,  consisting of elements preserving the standard symplectic form $\omega_{0}$. Furthermore,  $\operatorname{U}(2n,\Real)$ is  isomorphic to the unitary  group $\operatorname{U}(n,\Complex)$.  Thus, we may identify each Lagrangian plane $\Ln$  of $\Real^{2n}$ with a subspace of $\Complexn$ (which will be denoted by the same symbol $\Ln$).

A function $\varphi\in L_{\infty}(\Real^{2n})$ is said to be invariant under \emph{Lagrangian translations} ($\Ln$-\emph{invariant}) if for  every $h\in\Ln$
		$$\varphi(z-h)=\varphi(z), \quad  z=(x,y)\in\Real^{2n}.$$
As shown in \cite{Esmeral-Vasilevski}, the results on horizontal Toeplitz operators remain valid  for  Toeplitz operators with $\Ln$-invariant symbols.

	\section{ $k$-FC type measures and Toeplitz operators}
	\label{Preliminaries-2}
	In this section  we extend the results of \cite{Isra-Zhu,Vasilevski-Rozenblum,Zhu-F} about Fock-Carleson type measures, Fock-Carleson type measures for derivatives of order $k$, and Toeplitz operators defined by sesquilinear forms to the multidimensional case.
	\subsection{Fock-Carleson type measures}
	We  denote by $\operatorname{Borel}(\Complexn)$ the Borel $\sigma$-algebra of $\Complexn$ and by $\Borelpcn$  the set of all \emph{complex regular} Borel measures $\mu\colon\Borelcn\rightarrow\Complex$  with total variation $$|\mu|(B)=\sup\sum_{n\in\N}|\,\mu(B_{n})|,$$
with the supremum  taken over all Borel partitions ${B_{n}}$ of $B$, that satisfy the  conditions:
\begin{itemize}
	\item $|\mu|$ is locally finite: $|\mu|(\K)<\infty$ for each compact $\K\subset\Omega$;
	\item $\mu$ is regular, i.e.,
	\begin{align*}
	|\mu(A)|&=\sup\left\{|\mu(\K)|\colon\hspace{-2pt}\K\,\text{is  compact  $X\subset A$}\right\}
	=\inf\left\{|\mu(U)|\colon A\subset\,U\,\text{and $U$ is open}\right\}\hspace{-1pt}.
	\end{align*}
\end{itemize}
As it was mentioned in Introduction, for $\mu\in\Borelpcn$, the  Toeplitz operator $\bT_{\mu}$ with defining measure-symbol $\mu$ has the following integral representation
	\begin{equation}\label{Toeplitz-measure1}
	 (\bT_{\mu}f)(z)=\pi^{-n}\int_{\Complexn}e^{z\cdot\overline{w}}f(w)e^{-|w|^{2}}d\mu(w),\quad z\in\Complexn.
	\end{equation}
	Let $\mu\in\Borelpcnm$ be a complex regular Borel  measure, $\mu\in\Borelpcn$, satisfying  \eqref{cond-M}.  In this case integral  \eqref{Toeplitz-measure1}   converges and hence the operator $\bT_{\mu}$ is  well-defined on the dense subset of  finite linear combinations  of kernel functions.  In the special case of $d\mu=\varphi\,d\nu$, $\bT_{\mu}=\bT_{\varphi}$, see  \cite{Isra-Zhu,Zhu-F} for the complete bibliography.   From now on we will assume that $\mu$ satisfies \eqref{cond-M}.

	\smallskip
	It is well-known that the Berezin transform of a function  $\varphi\in\,L_{\infty}(\Complexn)$ coincides with  the  Berezin transform of the Toeplitz operator $\bT_{\varphi}$, and we will denote it by $\widetilde{\varphi}$.  The integral representation \eqref{int-rep-toeplit-u} of the Toeplitz operator $\bT_{\varphi}$ implies that
	 $$\widetilde{\varphi}(z)=\frac{1}{\pi^{n}}\int_{\Complexn}\varphi(w)e^{-|z-w|^{2}}d\nu_{2n}(w).$$
	This definition was extended to positive Borel measures in  \cite{Isra-Zhu}:
	\begin{equation}\label{berezin-mu}
\widetilde{\mu}(z)=\pi^{-n}\int_{\Complexn}|k_{z}(w)|^{2}e^{-|w|^{2}}d\mu(w)=\pi^{-n}\int_{\Complexn}e^{-|z-w|^{2}}d\mu(w),\quad z\in\Complexn,
	\end{equation}
	where
	\begin{equation}\label{normalized-kernel}
	 k_{z}(w)={K_{z}(w)}({K_{z}(z)})^{-1/2}=e^{w\cdot\overline{z}-\frac{|z|^{2}}{2}}
	\end{equation}
	is the \emph{normalized reproducing kernel} of $\Fock$. In particular, if $\bT_{\mu}$ is bounded on $\Fock$, then for every $z\in\Complex$, $\widetilde{\mu}(z)=\innerp{\bT_{\mu}k_{z}}{k_{z}},$
i.e., $\widetilde{\mu}$ is the Berezin transform of  $\bT_{\mu}$.
	
	\begin{definition}[\textbf{Fock-Carleson type measures}]
		A positive measure $\mu$ is said to be a \emph{Fock-Carleson} type measure for $\Fock$ (FC measure, in short), if there exists a constant $\omega(\mu)>0$ such that for every $f\in\Fock$
		\begin{equation*}
		 \int_{\Complexn}|f(w)|^{2}e^{-|w|^{2}}d\mu(w)\leq\,\omega(\mu)\,\|f\|_{\Fock}^{2}.
		\end{equation*}
	\end{definition}

	\noindent
	The next result provides a criterion for a Toeplitz operator $\bT_{\mu}$ with a positive measure  $\mu$  as defining symbol to be  bounded. For more details we refer the reader to \cite[Theorems 2.3 and 3.1]{Isra-Zhu}.
	\begin{proposition}\label{criterion-borel-measure}
		Let $\mu$ be a positive Borel regular measure on $\Complexn$.  Then the following conditions are equivalent:
		\begin{enumerate}
			\item  The Toeplitz operator $\bT_{\mu}$ is bounded on $\Fock$.
			\item The sesquilinear form \begin{equation*}
			 \operatorname{F}(f,g)=\pi^{-n}\int_{\Complexn}f(z)\overline{g(z)}e^{-|z|^{2}}d\mu(z)
			\end{equation*}
			is  bounded in $\Fock$.
			\item $\widetilde{\mu}$ is bounded on $\Complexn$.
			\item  For any fixed $\rb=(r_{j})_{j=1}^{n}$ with $r_{j}>0$,
 \begin{equation*}
			\mu(B_{\rb}(z))< C, \quad\text{for all $z\in\Complex$},
			\end{equation*}
			for some constant $C>0$, where $B_{\rb}(z)$ denotes the polydisk centered at $z$ with 'radius' $\rb$.
			\item  $\mu$ is a Fock-Carleson type measure.
		\end{enumerate}
	\end{proposition}
	\begin{proof} We  show here  1$\Rightarrow$2 and 2$\Rightarrow$3, the remaining cases can be found in \cite{Isra-Zhu,Zhu-F}.
		\noindent
		1$\,\Rightarrow$ 2. Assume that the  operator $\bT_{\mu}$ is bounded. By \eqref{Toeplitz-measure},  for every $f,g\in\Fock$,
		\begin{align*}		 \innerp{\bT_{\mu}f}{g}&=\pi^{-n}\int_{\Complexn}f(w)\innerp{K_{w}}{g}e^{-|w|^{2}}d\mu(w)=\pi^{-n}\int_{\Complexn}f(w)\overline{\innerp{g}{K_{w}}}e^{-|w|^{2}}d\mu(w)\\
		&=\pi^{-n}\int_{\Complexn}f(w)\overline{g(w)}e^{-|w|^{2}}d\mu(w)=\F(f,g).
		\end{align*}
		The boundedness of $\bT_{\mu}$ implies the statement.\\
		\noindent
		2$\,\Rightarrow$ 3. Assume that the sesquilinear form $\F$ is  bounded. For each $z\in\Complexn$,
		\begin{align*}
		 \widetilde{\mu}(z)&=\pi^{-n}\int_{\Complexn}e^{-|z-w|^{2}}d\mu(w)=e^{-|z|^{2}}\int_{\Complexn}|e^{\overline{z}\cdot w}|^{2}e^{-|w|^{2}}d\mu(w)=\F(k_{z},k_{z}),
		\end{align*}
		where $k_{z}$ is the normalized reproducing kernel in \eqref{normalized-kernel}. Therefore, $\widetilde{\mu}(z)\leq \|\F\|.$
	\end{proof}	
A natural generalization is to admit a complex valued Borel measure $\mu$ such that its variation $|\mu|$ is a FC measure. In such a case, as a by-product  of Proposition \ref{criterion-borel-measure},  the results of \cite{Isra-Zhu}  imply that the following norms for $\mu$ are equivalent:
	
	\begin{enumerate}
		\item $\|\mu\|_{1}=\|\bT_{\mu}\|$.
		\item $\displaystyle\|\mu\|_{2}=\sup_{z\in\Complex}\widetilde{|\mu|}(z)$.
		\item $\displaystyle\|\mu\|_{3}=\sup_{z\in\Complex}|\mu|(B_{\rb}(z))$, where $\rb$ is any fixed positive $n$-tuple.
		\item $\displaystyle\|\mu\|_{4}=\sup_{\stackrel{f\in\Fock}{ \|f\|=1}}\left\{\int_{\Complexn}|f(w)|^{2}e^{-|w|^{2}}d|\mu|(w)\right\}$.
	\end{enumerate}
\subsection{Fock-Carleson type measures for derivatives of order $k$}
\noindent
Rozenblum and Vasilevski \cite{Vasilevski-Rozenblum} introduced  Fock-Carleson type measures for derivatives of order $k\in\Entero$ ($k$-FC type measure, in short) for $\mathcal{F}^{2}(\Complex)$, and established some basic properties (Theorems 5.4 and 5.9). They also introduced Toeplitz operators defined by \emph{coderivatives} of $k$-FC measures. We extend here these definitions to $\Fock$ using the multi-index notation for derivatives. The proofs below are just natural modifications of the one-dimensional case.
\begin{proposition}\label{Bou-partial-k-f}
	Let $f\in\Fock$ and $k\in\Entero_{+}^{n}$. Then for every $z\in\Complexn$,
	\begin{equation*}
	|\partial^{k}\hspn f(z)|\leq C\,k! \|f\|_{\Fock}\prod_{j=1}^{n}(1+(\RE z_{j})^{2})^{k_{j}/2}(1+(\IM z_{j})^{2})^{k_{j}/2}e^{\frac{|z_{j}|^{2}}{2}}
	\end{equation*}
	with  a constant $C>0$ not depending on $k\in\Entero_{+}^{n}$.
\end{proposition}
\begin{proof}
	Let $f\in\Fock$ and $k\in\Entero_{+}^{n}$. For each $z=(z_{1},z_{2},\ldots,z_{n})\in\Complexn$, by Cauchy integral representation for several complex variables,  for any   polydisk $S_{z,\rb}=S_{z_{1},r_{1}}\times S_{z_{2},r_{2}}\times\ldots\times S_{z_{n},r_{n}}$ (where $S_{z_{k},r_{k}}=\{w_{k}\in\Complex\colon |w_{k}-z_{k}|<r_{k}\}$) and any $\rb=(r_{1},\ldots,r_{n})$ with $r_{j}>0$, the following equality holds
	$$\partial^{k} f(z)=(2i\pi)^{-n}k!\int_{\partial S_{z,\rb}} f(\zeta)\prod_{j=1}^{n}\frac{d\zeta_j}{(\zeta_{j}-z_j)^{k_j+1}} .$$
	Therefore, since $(r_{j}+|z_{j}|)-|\zeta_{j}|\geq0$ for any $\zeta=(\zeta_{1},\ldots,\zeta_{j})\in\,S_{z,r}$,
	\begin{align*}
	|\partial^{k}\hspn f(z)|&\leq (2\pi)^{-n}k!\hspace{-5pt}\int_{\partial S_{z,\rb}}\hspace{-5pt}|f(\zeta)|\prod_{j=1}^{n}\frac{d\zeta_{j}}{|z_{j}-\zeta_{j}|^{k_{j}+1}}\\
	&\leq (2\pi)^{-n}k!\prod_{j=1}^{n}r_{j}^{-(k_{j}+1)}\hspace{-5pt}\int_{\partial S_{z,\rb}}\hspace{-5pt}|f(\zeta)|\prod_{j=1}^{n}e^{\frac{(r_{j}+|z_{j}|)^{2}-|\zeta_{j}|^{2}}{2}}d\zeta_{j}\\
	&\leq (2\pi)^{-n}k!\left(\prod_{j=1}^{n}r_{j}^{-(k_{j}+1)}e^{\frac{(r_{j}+|z_{j}|)^{2}}{2}}(\pi r_{j}^{2})^{1/2}\right)\left(\int_{\partial S_{z,\rb}}|f(\zeta)|^{2}e^{-|\zeta|^{2}}d\zeta\right)^{1/2}\\
	&\leq 2^{-n}\pi ^{-n/2}k!\|f\|_{\Fock}\prod_{j=1}^{n}r_{j}^{-k_{j}}e^{\frac{(r_{j}+|z_{j}|)^{2}}{2}}.
	\end{align*}
	Taking $r_{j}=(1+x_{j}^{2})^{-1/2}(1+y_{j}^{2})^{-1/2}$, with $z_{j}=x_{j}+iy_{j}$,  we obtain
	\begin{align*}
	|\partial^{k}\hspn f(z)|&\leq2^{-n}\pi ^{-n/2}k!\|f\|_{\Fock}\prod_{j=1}^{n}(1+x_{j}^{2})^{k_{j}/2}(1+y_{j}^{2})^{k_{j}/2}e^{\frac{\left[1+(1+x_{j}^{2})^{1/2}(1+y_{j}^{2})^{1/2}|z_{j}|\right]^{2}}{2(1+x_{j}^{2})(1+y_{j}^{2})}}\\
	&=2^{-n}\pi ^{-n/2}k!\|f\|_{\Fock}\prod_{j=1}^{n}(1+x_{j}^{2})^{k_{j}/2}(1+y_{j}^{2})^{k_{j}/2}e^{\frac{1}{2(1+x_{j}^{2})(1+y_{j}^{2})}+\frac{|z_{j}|}{(1+x_{j}^{2})^{1/2}(1+y_{j}^{2})^{1/2}}+\frac{|z_{j}|^{2}}{2}}.
	\end{align*}
	Now, since  $\displaystyle\frac{|z_{j}|}{(1+x_{j}^{2})^{1/2}(1+y_{j}^{2})^{1/2}}\leq\hspace{-5pt} \sqrt{\frac{(x_{j}^{2}+1)+(y_{j}^{2}+1)}{(1+x_{j}^{2})(1+y_{j}^{2})}}\leq\hspace{-5pt}\sqrt{2}$ and\newline  $\displaystyle \frac{1}{2(1+x_{j}^{2})(1+y_{j}^{2})}\leq \frac{1}{2}$, we have
	\begin{align*}
	|\partial^{k}\hspn f(z)|&\leq 2^{-n}\pi ^{-n/2}e^{\frac{2\sqrt{2}+1}{2}}k!\|f\|_{\Fock}\prod_{j=1}^{n}(1+x_{j}^{2})^{k_{j}/2}(1+y_{j}^{2})^{k_{j}/2}e^{\frac{|z_{j}|^{2}}{2}}.\qedhere
	\end{align*}
\end{proof}
\begin{definition}[\textbf{$k$-FC measures}] \label{defin-kFC}
	Let $k\in\Entero_{+}^{n}$. A  positive  Borel regular measure $\mu$ is called  a \emph{Fock-Carleson} type measure for derivatives of order $k$ ($k$-FC, in short) for $\Fock$  if there exists $\omega_{k}(\mu)>0$ such that for every $f\in\Fock$
	\begin{equation}\label{ineq-kFC}
	\int_{\Complexn}\left|\partial^{k}\hspn f(w)\right|^{2}e^{-|w|^{2}}d\mu(w)\leq\,\omega_{k}(\mu)\,\int_{\Complexn}|f(w)|^{2}e^{-|w|^{2}}d\nu_{2n}(w).
	\end{equation}
	If $|k|=0$, then any $0$-FC type measure is just a FC-measure  for $\Fock$.   A complex regular Borel measure is called $k$-FC if \eqref{ineq-kFC} is satisfied for $\mu$ replaced by $|\mu|$.
\end{definition}

Denote by $\mu_{p}$ the Borel measure on $\Complexn$ given by
\begin{equation}\label{measure-muk}
\mu_{p}(B)=\int_{B}\prod_{j=1}^{n}(1+x_{j}^{2})^{p_{j}}(1+y_{j}^{2})^{p_{j}}d\mu(x+iy),\quad B\subset\Complexn.
\end{equation}
The following results relate the $k$-FC and Fock-Carleson type measures for $\Fock$. The proofs are almost literally the same as  \cite[Theorem 5.4 and Corollary 5.5]{Vasilevski-Rozenblum}. It is enough to replace Proposition 5.1 of \cite{Vasilevski-Rozenblum} by the above Proposition \ref{Bou-partial-k-f} and take  the product of the lattices used  in the proof of \cite[Theorem 5.4]{Vasilevski-Rozenblum}.
\begin{proposition}\label{prop-crmur}
	Let $k\in\Entero_{+}^{n}$. A positive measure $\mu$ is a $k$-FC type measure if and only if, for some (and, therefore for any) $\rb>0$, the following quantity is finite:
	\begin{equation*}
	C_{k}(\mu,\rb)=(k!)^{2}\sup_{z\in\Complexn}\mu_{k}(B_{\rb}(z)).
	\end{equation*}
	For a fixed $\rb$, the constant $\omega_{k}(\mu)$ in \eqref{ineq-kFC} can be taken as $\omega_{k}(\mu)=C(\rb)C_{k}(\mu,\rb)$ where $C(\rb)$  depends only on $\rb$. For a complex measure $\mu$, the 'if' part holds true.
\end{proposition}

\begin{remark}
	The explicit  dependence of the  constant on $k$ was crucial in \cite{Vasilevski-Rozenblum},  where symbols involving derivatives of unbounded order were considered. Our formulation is oriented to considering such general symbols in the forthcoming research.
\end{remark}
\begin{proposition}\label{prop-Ck}
	For any  $p,k\in\Entero_{+}^{n}$,   a positive Borel measure $\mu$ is a $k$-FC type measure if and only if the measure $\mu_{k-p}$ is a $p$-FC type measure. Furthermore,
	$C_{k-p}(\mu_{p},\rb)=C_{k}(\mu,\rb)$.
\end{proposition}

The concept of $k$-FC type measure can be extended to half positive integer multi-indices $k$ by means of \eqref{measure-muk} and Proposition \ref{prop-Ck}.

\begin{definition}[\textbf{$k$-FC measures: extended version}]
	If $k\in(\mathbb{Z}_{+}/2)^n$, then we say that $\mu$ is a $k$-FC type measure  if  the quantity
	\begin{equation*}
	C_{k}(\mu)=(k!)^{2}\sup_{z\in\Complex}|\mu_{k}|(B_{\pmb{\sqrt{n}}}(z))
	\end{equation*}
	is finite.  Here $B_{\pmb{\sqrt{n}}}(z)$ denotes the polydisk in $\Complexn$ centered at $z=(z_{1},\ldots,z_{n})$ and radius $\rb=\pmb{\sqrt{n}}:=(1,1,\ldots,1)$, $|\pmb{\sqrt{n}}|=\sqrt{n}$  (this  radius is taken here just for convenience.)	
\end{definition}

\begin{example}
	Note that any measure with compact support is a $k$-FC type measure for each $k$. On the other hand, given $k\in (\mathbb{Z}_{+}/2)^n$,  the  Borel measure
	\begin{equation*}
	 d\mu(z)=\prod_{j=1}^{n}\frac{dx_{j}dy_{j}}{(1+x_{j}^{2})^{k_{j}}(1+y_{j}^{2})^{k_{j}}}
	\end{equation*}
	is, by Proposition \ref{prop-Ck},  a $k$-FC type measure for $\Fock$. In fact,
	 $$\mu_{k}(B)=\int_{B}\prod_{j=1}^{n}(1+x_{j}^{2})^{k_{j}}(1+y_{j}^{2})^{k_{j}}d\mu(x+iy)=\nu_{2n}(B).$$
	for every Borel set $B\subset\Complexn$. Here  $\mu_{k}=\nu_{2n}$ is the usual Lebesgue measure on $\Complexn$. Therefore $\mu_{k}$ is a Fock-Carleson type measure for $\Fock$.
\end{example}


\subsection{Toeplitz operators generated by  $k$-FC measures}
\noindent
  Toeplitz operators generated by  bounded sesquilinear forms $\F$ on $\mathcal{F}^2(\mathbb{C})$ were introduced in \cite{Vasilevski-Rozenblum}. The approach proposed there was as follows (see \cite{Vasilevski-Rozenblum} for more details).

Let $\h$ be a Hilbert space of functions defined on a domain $\Omega\subset\Complex^{n}$, and let $\A$ be a reproducing kernel Hilbert subspace of $\h$ with the reproducing kernel function $K_{z}$ at the point $z\in\Omega$. The orthogonal projection $\Pn$ of $\h$ onto $\A$ has thus the form
$(\Pn f)(z)=\innerp{f}{K_{z}}$.

Let $\F (\cdot,\cdot)$ be a bounded sesquilinear form in $\A$.  Then by the Riesz theorem for bounded sesquilinear forms, there exists a unique bounded linear operator $\mathbf{T}$ in $\A$ such that $\F (f, g) = \langle \mathbf{T} f, g\rangle$, for all $f, g \in \A$.
Now the Toeplitz operator $\bT_{\F}$ generated  by a bounded sesquilinear form $\F$ and acting on $\A$ is defined as
\begin{equation}\label{Toeplitz-sesquilinear-form}
(\bT_{\F}f)(z)=\F(f,K_{z})=\innerp{\mathbf{T}f}{K_{z}}=(\mathbf{T}f)(z),\quad z\in\Omega.
\end{equation}	

In what follows we assume that $\h=L_{2}(\Complexn,e^{-|w|^{2}}d\nu_{2n}(w))$ and $\A=\Fock$.  Our aim is to consider  Toeplitz operators generated by sesquilinear forms given by $k$-FC type measures for $\Fock$.

Let $\mu$ be a $k$-FC type measure for $\Fock$, where $k\in (\mathbb{Z}_{+}/2)^{n}$. For $\alpha,\, \beta\in\Entero_{+}^{n}$, with $2k=\alpha+\beta$, we extend to the $n$-dimensional case the definition of  the \emph{coderivative} $\ppartial^{\alpha}\overline{\ppartial}{}^{\beta}\mu$ introduced  in \cite{Vasilevski-Rozenblum} as follows: for a function $h=f\overline{g}\in\,L_{1}(\Complexn,e^{-|w|^{2}}d\nu_{2n}(w))$ with $f,g\in\Fock$
\begin{equation*}
(\ppartial^{\alpha}\overline{\ppartial}{}^{\beta}\mu,h)=(-1)^{\alpha+\beta}(\mu\,G,\partial^{\alpha}\overline{\partial}{}^{\beta}h)=(\mu\,G,\partial^{\alpha}\hspn f\overline{\partial^{\beta}g}),\quad G(z)=e^{-|z|^{2}},
\end{equation*}
(where $(\cdot,\cdot)$ is the intrinsic pairing between measures and functions), provided that the right-hand side makes sense.

The  sesquilinear form $\F_{\mu,\alpha,\beta}$ on  $\Fock$ associated with the \emph{coderivative} $\ppartial^{\alpha}\overline{\ppartial}{}^{\beta}\mu$ is given by	
\begin{equation}\label{ses-F-al-bet}
\hspace{-8pt}\F_{\mu,\alpha,\beta}(f,g)=(\ppartial^{\alpha}\overline{\ppartial}{}^{\beta}\mu,f\overline{g})=\pi^{-n}\int_{\Complexn}\partial^{\alpha}\hspn f(z)\overline{\partial^{\beta}g(z)}e^{-|z|^{2}}d\mu(z).
\end{equation}

For $\alpha,\beta\in\Entero_{+}^{n}$, $k\in (\mathbb{Z}_{+}/2)^{n}$, $2k=\alpha+\beta$, and the coderivative $\ppartial^{\alpha}\overline{\ppartial}{}^{\beta}\mu$ of a $k$-FC type measure $\mu$, we define the Toeplitz operator $\bT_{\ppartial^{\alpha}\overline{\ppartial}{}^{\beta}\mu}$ as the Toeplitz operator generated by the sesquilinear form \eqref{ses-F-al-bet} by means of \eqref{Toeplitz-sesquilinear-form}, i.e., for $ f\in\Fock$,
\begin{equation*}
(\bT_{\ppartial^{\alpha}\overline{\ppartial}{}^{\beta}\mu}f)(z)=\F_{\mu,\alpha,\beta}(f,K_{z})=\pi^{-n}z^{\beta}\int_{\Complexn}\partial^{\alpha}\hspn f(w)e^{z\cdot\overline{w}}e^{-|w|^{2}}d\mu(w).
\end{equation*}
Next, we introduce the Berezin transform of the coderivatives $\ppabeta\mu$ of a $k$-FC type measure by
\begin{equation}\label{berezin-a-b-mu}
\widetilde{\ppabeta\mu}(z)=z^{\beta}\overline{z}^{\alpha}\int_{\Complexn}e^{-|z-w|^{2}}d\mu(w),\quad z\in\Complexn.
\end{equation}
If the Toeplitz operator $\bT_{\ppabeta\mu}$ is bounded, then $\widetilde{\ppabeta\mu}=\widetilde{\bT_{\ppabeta\mu}}$, i.e.,
\begin{equation*}
\widetilde{\ppabeta\mu}(z)=\frac{\innerp{\bT_{\ppabeta\mu}K_{z}}{K_{z}}}{\innerp{K_{z}}{K_{z}}}=e^{-|z|^{2}}\F_{\mu,\alpha,\beta}(k_{z},k_{z}),\quad z\in\Complexn.
\end{equation*}
The proof of the following result is literally the same as of Proposition 6.4 in \cite{Vasilevski-Rozenblum}, we only  need to replace Corollary 5.5 in \cite{Vasilevski-Rozenblum} by Proposition \ref{prop-Ck}.
\begin{proposition}\label{prop-boundedness-Fab}
	Let $\alpha,\beta\in\Entero_{+}^{n}$ and  $k\in (\mathbb{Z}_{+}/2)^{n}$ satisfying	 $\alpha+\beta=2k$. 	If	 $\mu\in\Borelpcn$ is a positive  $k$-FC type measure for $\Fock$ then the associated sesquilinear form $\F_{\mu,\alpha,\beta}$ given in \eqref{ses-F-al-bet}   is bounded.
\end{proposition}

\begin{remark}
	Let  $\mu$ be a  $k$-FC measure for $\Fock$ and let  $\mu_R$ be  the measure obtained from $\mu$ by restricting to the exterior of the ball of the radius $R$ and centered in $0$. Observe that,  by Proposition  \ref{prop-boundedness-Fab}, the Toeplitz operator $\bT_{\ppabeta\mu}$ is bounded, in addition,  if the constant $C_k(\mu_R) \to 0$,
as $R\to\infty$,   then $\bT_{\ppabeta\mu}$ is compact.
\end{remark}
\noindent
If $\mu\in\Borelpcn$ is a positive  $k$-FC type measure for $\Fock$, by Proposition \ref{prop-boundedness-Fab}, there exists a unique bounded operator $\mathbf{A}_{\alpha,\beta}$ such that  $\F_{\mu,\alpha,\beta}(f,g)=\innerp{\mathbf{A}_{\alpha,\beta}f}{g}$ for every $f,g\in\Fock$. The Toeplitz operator $\bT_{\ppabeta\mu}$ is bounded since $\F_{\mu,\alpha,\beta}$ is bounded and  $\mathbf{A}_{\alpha,\beta}=\bT_{\ppabeta\mu}$. In effect, 	for each $g\in\Fock$, by the reproducing property and the derivation under integral sign,  for all $z\in\Complexn$:
\begin{equation}\label{aux-deriv}
\partial^{ \alpha}\hspn g(z)=\pi^{-n}\int_{\Complexn}g(w)\overline{w}^{\alpha}e^{z\cdot\overline{w}}e^{-|w|^{2}}d\nu_{2n}(w).
\end{equation}
Thus, by Fubini's theorem, \eqref{ses-F-al-bet},  and  \eqref{aux-deriv},  for every $f,g\in\Fock$,
\begin{align}\label{Fabeta=Tpamu}
\innerp{\bT_{\ppabeta\mu}f}{g}
&=\pi^{-n}\int_{\Complexn}\left(\pi^{-n}z^{\beta}\int_{\Complexn}\partial^{\alpha}\hspn f(w)e^{z\cdot\overline{w}}e^{-|w|^{2}}d\mu(w)\right)\overline{g(z)}e^{-|z|^{2}}d\nu_{2n}(z)\nonumber\\
&=\pi^{-n}\int_{\Complexn}\partial^{\alpha}\hspn f(w)\overline{\left(\pi^{-n}\int_{\Complexn}\overline{z}^{\beta}e^{w\cdot\overline{z}}g(z)e^{-|z|^{2}}d\nu_{2n}(z)\right)}e^{-|w|^{2}}d\mu(w)\nonumber\\
&=\pi^{-n}\int_{\Complexn}\partial^{\alpha}\hspn f(w)\overline{\partial^{\beta}\hspn g(w)}e^{-|w|^{2}}d\mu(w)=\F_{\mu,\alpha,\beta}(f,g).
\end{align}

\section{Horizontal Toeplitz operators with  FC type measures as symbols}\label{Section-Hor-FC}
Given a complex regular Borel  measure $\varrho\in\Borelprn$ we denote by $\mu=\varrho\otimes\eta$ the tensor product of the measures $\varrho$ and    $\eta$ on $\Realn$.  i.e.,   for any $A, B\in\Borelrn$, $\mu(A\times B)=\varrho(A)\eta(B),$ with the usual extension to all Borel sets in $\mathbb{R}^{2n}$
\begin{definition}[\textbf{Horizontal measures}]\label{hor-meas}
 We say that $\mu\in\Borelpcn$ is \emph{horizontal} if    $\mu=\varrho\otimes\nu_{n}$ for some $\varrho\in\Borelprn$, where $\nu_{n}$ is the  Lebesgue measure on $\Realn$.
Furthermore, if  $\mu=\varrho\otimes\nu_{n}$ is an FC type measure for $\Fock$ we say that $\mu$ is an \emph{hFC}.
\end{definition}
\noindent
The following lemma is  analogous to the injectivity property of the  Berezin transform for Toeplitz operators $\bT_{\mu}$ with FC type measures as symbols. Although for an immediate use it is sufficient to prove it for $k=0$, we admit any  $k\in(\Entero_{+}/2)^{n}$ since it will be required further on in Theorem \ref{thm-char-2kfc}.
\begin{lemma}\label{lemma-medida-zkzero}
	Let  $k\in(\Entero_{+}/2)^{n}$ and $\mu\in\Borelpcn$ be a complex  measure that satisfies the $(M)$-condition \eqref{cond-M}. If
	$$(\RE \,z)^{2k}\hspace{-4pt}\int_{\Complexn}e^{-|z-w|^{2}}d\mu(w)=0\quad \text{or}\quad (\IM \, z)^{2k}\hspace{-4pt}\int_{\Complexn}e^{-|z-w|^{2}}d\mu(w)=0$$ for any  $z$ in $\Complexn$ then $\mu$ is the zero measure.
\end{lemma}
\begin{proof}
	Suppose that $\displaystyle (\RE \,z)^{2k}\hspace{-4pt}\int_{\Complexn}e^{-|z-w|^{2}}d\mu(w)=0$ for all  $z$ in $\Complexn$. 	 Let $\Psi\colon \Complexn\times\overline{\Complexn}\rightarrow\Complex$ be the mapping
	\begin{align*}
	\Psi(z,w)&=\sum_{\beta\leq 2k}\left(\begin{matrix}
	2k\\ \beta
	\end{matrix}\right)\int_{\Complexn}w^{2k-\beta}e^{w\cdot \zeta}z^{\beta}e^{z\cdot\overline{\zeta}}e^{-|\zeta|^{2}}d\mu(\zeta)\\
	&=(w+z)^{2k}\hspace{-4pt}\int_{\Complexn}e^{w\cdot \zeta}e^{z\cdot\overline{\zeta}}e^{-|\zeta|^{2}}d\mu(\zeta).
	\end{align*}
	Note that for any triangle $\Delta$ in $\Complexn$ by the Fubini's Theorem,
	\begin{align*}
	\int_{\partial\Delta}\Psi(z,w)dz&=\sum_{\beta\leq 2k}\left(\begin{matrix}
	2k\\ \beta
	 \end{matrix}\right)\int_{\Complexn}w^{2k-\beta}e^{w\cdot\zeta}\left(\int_{\partial\Delta}z^{\beta}e^{z\cdot\overline{\zeta}}dz\right)d\mu(\zeta)=0,\\
	\int_{\partial\Delta}\Psi(z,w)dw&=\sum_{\beta\leq 2k}\left(\begin{matrix}
	2k\\ \beta
	 \end{matrix}\right)\int_{\Complexn}z^{\beta}e^{z\cdot\overline{\zeta}}\left(\int_{\partial\Delta}w^{2k-\beta}e^{w\cdot \zeta}dw\right)d\mu(\zeta)=0.
	\end{align*}
	Thus, by  Hartogs's Theorem, $\Psi$ is an analytic function in $\Complexn\times\overline{\Complexn}$  and $$\Psi(z,\overline{z})=e^{|z|^{2}}2^{|2k|}(\RE z)^{2k}\int_{\Complexn}e^{-|\zeta-z|^{2}}d\mu(\zeta)=0.$$
	Then, by \cite[Proposition 1.69]{Folland2}, $\Psi\equiv0$. Therefore, the function $\Phi\colon \Complexn\times\overline{\Complexn}\rightarrow\Complex$,
	 $$\Phi(z,w)=\int_{\Complex}e^{z\cdot\zeta}e^{w\cdot\zeta}e^{-|\zeta|^{2}}d\mu(\zeta),$$
	satisfies  $\Phi(z,w)=0$ for all $z\neq-w.$ However, by the same argument applied to $\Psi,$ the function $\Phi$ is analytic in $\Complexn\times\overline{\Complexn}$ and hence continuous. Thus, $\Phi\equiv0$ in $\Complexn\times\overline{\Complexn}$.
	
	In particular, if $d\varsigma(u,v)=e^{-|u+iv|^{2}}d\mu(u+iv)$ then  for any $x,y\in\Realn$,
	\begin{align*}
	 \int_{\Real^{n}\times\Realn}e^{-i(x,y)\cdot(u,v)}d\varsigma(u,v)&=\int_{\Real^{n}\times\Realn}e^{(-y+ix)\cdot(u+iv)+(y+ix)\cdot(u-iv)}d\varsigma(u,v)\\
	&=\int_{\Complex^{n}} e^{(-y+ix)\cdot\zeta+(y+ix)\cdot\overline{\zeta}}e^{-|\zeta|^{2}}d\mu(\zeta)=\Phi(y+ix,-y+ix)=0,
	\end{align*}
	i.e., the Fourier-Stieltjes transformation of the bounded complex  measure $\varsigma$ in $\Realn\times\Realn$ is 0. Thus by the injectivity of Fourier-Stieltjes transform,  $\varsigma\equiv0$, see \cite[Proposition 3.8.6]{Bogachev}, and hence $\mu\equiv0$.	
	\noindent
	Now, if $\displaystyle (\IM \,z)^{2k}\hspace{-4pt}\int_{\Complexn}e^{-|z-w|^{2}}d\mu(w)=0$  for any  $z$ in $\Complexn$, then the proof of the statement follows almost literally the above reasoning, we only need to replace $\Psi $ by the function
	\begin{align*}
	\psi(z,w)&=\sum_{\beta\leq 2k}\left(\begin{matrix}
	2k\\ \beta
	\end{matrix}\right)(-1)^{\beta}\int_{\Complexn}w^{2k-\beta}e^{w\cdot \zeta}z^{\beta}e^{z\cdot\overline{\zeta}}e^{-|\zeta|^{2}}d\mu(\zeta)\\
	&=(w-z)^{2k}\hspace{-4pt}\int_{\Complexn}e^{w\cdot \zeta}e^{z\cdot\overline{\zeta}}e^{-|\zeta|^{2}}d\mu(\zeta).\qedhere
	\end{align*}
\end{proof}
\begin{lemma}\label{lemma-sep-int-ber}
	Let $\varrho$ be a complex regular measure on $\mathbb{R}^n$ such that $\mu=\varrho\otimes\nu_{n}\in\Borelpcnm$. Then for every $z=x+iy\in\Complexn$,
	\begin{equation}\label{berezin-convo}
	 \int_{\Complexn}e^{-|z-w|^{2}}d\mu(w)=\pi^{-n/2}\int_{\Realn}e^{-(t-x)^{2}}\,d\varrho(t).
	\end{equation}
\end{lemma}
\begin{proof}
	Let $z=x+iy\in\Complexn$. Then
	\begin{align*}
	 \int_{\Complexn}e^{-|z-w|^{2}}d\mu(w)&=\pi^{-n}\int_{\Complexn}e^{-|x+iy-w|^{2}}d\mu(w)=\pi^{-n}\int_{\Realn}\int_{\Realn}e^{-|(x-t)+i(y-v)|^{2}}d\varrho(t)d\nu_{n}(v)\\
	 &=\pi^{-n}\int_{\Realn}\int_{\Realn}e^{-(x-t)^{2}-(y-v)^{2}}d\varrho(t)d\nu_{n}(v)\\
	 &=\pi^{-n}\left(\int_{\Realn}e^{-(y-v)^{2}}d\nu_{n}(v)\right)\left(\int_{\Realn}e^{-(x-t)^{2}}d\varrho(t)\right)\\
	&=\pi^{-n/2}\int_{\Realn}e^{-(x-t)^{2}}d\varrho(t).\qedhere
	\end{align*}
\end{proof}
	\begin{theorem}[\textbf{Criterion for hFC measures}]\label{thm-char-hfc}
	Let $\mu$ be a positive FC measure for $\Fock$. Then the following conditions are equivalent:
	\begin{enumerate}
		\item $\bT_{\mu}$ is horizontal.
		\item  $\widetilde{\mu}$ depends only on $\RE \, z$.
		\item  $\mu$ is invariant under horizontal translations, i.e., for every Borel set $X\subset \Complexn$ and every $h\in\Realn$, $$\mu(X+ih)=\mu(X).$$
		\item For every Borel sets $Y,Z\subset\Realn$, and every $h\in\Realn$,
		$$\mu(Y\times (Z+h))=\mu(Y\times Z)$$
		\item $\mu$ is a horizontal measure, i.e., there exists  $\varrho\in\Borelprn$  such that $\mu=\varrho\otimes\nu_{n}.$
	\end{enumerate}
\end{theorem}
\begin{proof}
	$1\Leftrightarrow2$.  If the measure $\mu$ is a FC for $\Fock$ then the Toeplitz operator $\bT_{\mu}$ is bounded. Therefore, by Proposition \ref{criterio-horizontal-operators}, this operator is horizontal if and only if the corresponding Berezin transform depends only on $\RE \,z$.\\
	\smallskip
	\noindent
	$2\Rightarrow3$.  Let $h\in\Realn$,  $X\in\Borelcn$ and set $\mu_{h}(X)=\mu(X+ih)$. For every $z\in\Complexn$,
	\begin{equation*}
	\widetilde{ \mu_{h}}(z)=\int_{\Complexn}\hspace{-7pt}e^{-|z-w|^{2}}d\mu(w+ih)=\int_{\Complexn}\hspace{-7pt}e^{-|(z+ih)-w|^{2}}d\mu(w)
	=	\widetilde{\mu_{h}}(z+ih),\quad
	\end{equation*}
	for any $f\in\Fock$.Therefore,  for every $z\in\Complexn$,	
\begin{align*}\label{aux-lambdah}
	\widetilde{\lambda_{h}}(z)&=\int_{\Complexn}e^{-|z-w|^{2}}d\lambda_{h}(w)=0,
	\end{align*}
	where  $\lambda_{h}=\mu_{h}-\mu$. 	This is  a signed measure such that $|\lambda_{h}|(B)\leq |\mu_{h}|(B)+|\mu|(B)$ for all $B\in\Borelcn$,  $|\lambda_{h}|$ satisfies the $(M)$-condition, and for every $f\in\Fock$
	\begin{align*}
	\int_{\Complex^{n}}|f(z)|^{2}e^{-|z|^{2}}d|\lambda_{h}|(z)&\leq \int_{\Complex^{n}}|f(z)|^{2}e^{-|z|^{2}}d|\mu_{h}|(z)+\int_{\Complex^{n}}|f(z)|^{2}e^{-|z|^{2}}d|\mu|(z)\\
	 &=\int_{\Complex^{n}}|f(z-ih)|^{2}e^{-|z-ih|^{2}}d|\mu|(z)+\omega(\mu)\|f\|^{2}\\
	&=\int_{\Complex^{n}}|e^{ih\cdot z+\frac{h^{2}}{2}}\W_{ih}f(z)|^{2}e^{-|z-ih|^{2}}d|\mu|(z)+\omega(\mu)\|f\|^{2}\\
	&=\int_{\Complex^{n}}|e^{ih\cdot z}|^{2}e^{h^{2}}|\W_{ih}f(z)|^{2}e^{-(|z|^{2}+z\cdot ih-ih\cdot\overline{z}+h^{2})}d|\mu|(z)+\omega(\mu)\|f\|^{2}\\
	&=\omega(\mu)\|\W_{-ih}f\|^{2}+\omega(\mu)\|f\|^{2}.
	\end{align*}
	Since $\W_{-ih}$ is unitary on the Fock space, this shows that $\lambda_{h}$ is a  FC  measure for $\Fock$. Now, by Lemma  \ref{lemma-medida-zkzero}, with $|k|=0$, we see that $\lambda_{h}$ is the zero measure, i.e., for every Borel set  $X\subset\Complexn$, $\mu(X+ih)=\mu(X)$   for every $h\in\Realn$.
	
	\noindent
	$3\Rightarrow4$.  It is immediate.\\
	$4\Rightarrow5$.  Suppose that  $\mu(Y\times (Z+h))=\mu(Y\times Z)$ for any Borel sets $Y,Z\in\Borelrn$ and every $h\in\Realn$.  We define the mapping  $\Psi_{Y}\colon \Borelrn\to[0,+\infty]$  by  $\Psi_{Y}(Z)=\mu(Y\times Z)$. It is easily seen that $\Psi_{Y}$ is a   measure on $\Realn$ such that $\Psi_{Y}(C)<+\infty$ for any compact subset $C$ in $\Realn$ since $\mu$ is regular. Hence $\Psi_{Y}$ is regular by \cite[Theorem 2.18]{Rudin}. Furthermore, $\Psi_{Y}$ is  translation invariant.  Consequently,  $\Psi_{Y}$ is a  multiple of the Lebesgue measure $\nu_{n}$, \cite[2.20 Theorem (d)]{Rudin}, i.e., there exists a number $\varrho(Y)$ such that $\mu(Y\times Z)=\varrho(Y)\nu_{n}(Z)$ for every $Z\in\Borelrn$.  Now,  $\varrho$ is a positive regular Borel measure on $\Realn$  since $\mu(Y\times [0,1]^{n})=\varrho(Y)$ for each set $Y\in\Borelrn$.  From this, it follows  that  $\mu$ is horizontal.\\
	$5\Rightarrow1$. 	If $\mu=\varrho\otimes\nu_{n}$ is a hFC type measure for $\Fock$, then by   \eqref{berezin-mu}  and Lemma \ref{lemma-sep-int-ber},   the Berezin transform  $\widetilde{\mu}$ of the Toeplitz operator $\bT_{\mu}$ is
	\begin{equation}\label{Berezin-h-mu}
	\widetilde{\mu}(z)=\pi^{-n/2}\int_{\Realn}e^{-(x-u)^{2}}d\varrho(u).
	\end{equation}
	Therefore, by Proposition \ref{criterio-horizontal-operators} the Toeplitz operator  $\bT_{\mu}$ is horizontal.
\end{proof}
\noindent
Since a complex regular Borel measure $\mu$ is  a FC	type measure  if and only if  its variation $|\mu|$ is a FC measure, it follows  that Theorem \ref{thm-char-hfc} remains valid for such type of measures. Next, we show that every Toeplitz operator with horizontal measure as symbol is unitarily equivalent to the multiplication operator by some  $L_{\infty}(\Realn)$-function.
\begin{theorem}[\textbf{Diagonalization of Toeplitz operators}]\label{diag-To-hFC}
	Let  $\mu= \varrho\otimes\nu_{n}$ be a hFC measure for $\Fock$. Then the Toeplitz operator $\bT_{\mu}$ is unitarily equivalent to \linebreak $\Bb \bT_{\mu}\Bb^{*}=\gamma_{\varrho}\mathrm{Id},$
	where the function $\gamma_{\varrho}\colon\Realn\rightarrow\Complex$ is given by the formula
	\begin{equation*}
	 \gamma_{\varrho}(x)=\left(\frac{2}{\pi}\right)^{n/2}\int_{\Realn}e^{-(x-\sqrt{2}y)^{2}}d\varrho(y),\quad x\in\Realn.
	\end{equation*}
\end{theorem}
\begin{proof}
	Let $S=\Bb^{*}M_{\gamma_{\varrho}}\Bb$. Then, by \eqref{Bargmann-inversa-kernel}, \eqref{normalized-kernel} and Tonelli's theorem, a simple computation shows that for every $z=u+iv\in\Complex$
	\begin{align*}
	\widetilde{S}(u+iv)&=\innerp{\Bb^{*} M_{\gamma_{\varrho}}\Bb k_{u+iv}}{k_{u+iv}}=	 \frac{2^{n/2}}{\pi^{n}}\int_{\Realn}\left(\int_{\Realn}e^{-(x-\sqrt{2}y)^{2}}e^{-(\sqrt{2}u-x)^{2}}dx\right)d\varrho(y)\\
&	 =\left(\frac{1}{\pi^{n}}\int_{\Realn}e^{-(y-u)^{2}}d\varrho(y)\right)\left(\int_{\Realn}e^{-(t-u)^{2}}dt\right)
	=\pi^{n/2}\int_{\Realn}e^{-(y-u)^{2}}d\varrho(y).
	\end{align*}
	By \eqref{Berezin-h-mu}, the last integral is the Berezin transform   $\widetilde{ \mu}$ of  $\bT_{\mu}$,  where  $\mu=\varrho\otimes\nu_{n}$. Therefore, by the injectivity of the Berezin transform,  $S=\bT_{\mu}$, with $\mu= \varrho\otimes \nu_{n}$.
\end{proof}
\noindent
As a consequence of  Theorems \ref{thm-char-hfc} and \ref{diag-To-hFC} we have the following result.
\begin{corollary}\label{coro-Linf-FC}
	Let $\varrho\in\Borelprn$. Then  $\gamma_{\varrho}\in\,L_{\infty}(\Realn)$ if and only if $\varrho\otimes\nu_{n}$ is a FC  measure for $\Fock$.
\end{corollary}
	
	\section{Horizontal Toeplitz operators  generated by  $k$-FC type measures}\label{Section-Hor-kFC}	
	
	Let $k\in(\Entero_{+}/2)^{n}$. If $\mu$ is a $k$-FC type measure for $\Fock$, then, by Proposition \ref{prop-boundedness-Fab},   the Toeplitz operator $\bT_{\ppabeta \mu} $ generated by the sesquilinear form $\F_{\mu,\alpha,\beta}$ is bounded for $\alpha,\beta\in\Entero_{+}^{n}$  such that $2k=\alpha+\beta$. Thus, the sesquilinear form
	\begin{equation}\label{ses-fom-Fmu2k}
	\Frc_{\mu,2k}=\sum_{\beta\leq 2k}\left(\begin{matrix}
	2k\\ \beta
	\end{matrix}\right)\F_{\mu,2k-\beta,\beta}
	\end{equation}
	is bounded as soon as	the Toeplitz operator  $\bT_{\Frc_{\mu,2k}}$ is, moreover, the action of the latter operator is given by
	\begin{equation*}
	\bT_{\Frc_{\mu,2k}}=\sum_{\beta\leq 2k}\left(\begin{matrix}
	2k\\ \beta
	\end{matrix}\right) \bT_{\ppartial^{2k-\beta}\overline{\ppartial}{}^{\beta}\mu}.
	\end{equation*}
	By \eqref{ses-fom-Fmu2k},  the action of the  Toeplitz operator  $\bT_{\Frc_{\mu,2k}}$ on $f\in\Fock$ has the following integral representation
	\begin{equation*}
	(\bT_{\Frc_{\mu,2k}}f)(z)=\pi^{-n}\sum_{\beta\leq 2k}\left(\begin{matrix}
	2k\\ \beta
	 \end{matrix}\right)z^{\beta}\int_{\Complexn}(\partial^{2k-\beta}f)(w)e^{z\cdot\overline{w}}e^{-|w|^{2}}d\mu(w),\quad z\in\Complexn.
	\end{equation*}
	Following \eqref{ses-F-al-bet}, we introduce the \emph{real coderivative} $\ppartial_{R}^{2k}\mu$ of $\mu$ by
	\begin{equation*}
	\ppartial_{R}^{2k}\mu=\sum_{\beta\leq 2k}\left(\begin{matrix}
	2k\\ \beta
	\end{matrix}\right) \ppartial^{2k-\beta}\overline{\ppartial}{}^{\beta}\mu=(\ppartial+\overline{\ppartial})^{2k}\mu.
	\end{equation*}
	Due to  \eqref{Fabeta=Tpamu},  $\bT_{\ppartial_{R}^{2k}\mu}:=\bT_{\Frc_{\mu,2k}}$ is the unique bounded operator such that
	$$\innerp{\bT_{\ppartial_{R}^{2k}\mu}f}{g}=\Frc_{\mu,2k}(f,g).$$
	 for any $f,g\in\Fock$. Now, by \eqref{berezin-a-b-mu},  the corresponding Berezin transform $\widetilde{\ppartial_{R}^{2k}\mu}$  equals
	\begin{align}\label{Bere-real-code-mu}
	\widetilde{\ppartial_{R}^{2k}\mu}(z)&=\pi^{-n}\sum_{\beta\leq 2k}\left(\begin{matrix}
	2k\\ \beta
	 \end{matrix}\right)z^{\beta}\overline{z}^{2k-\beta}\hspace{-4pt}\int_{\Complexn}\hspace{-4pt}e^{-|z-w|^{2}}d\mu(w)\nonumber\\
	&=2^{|2k|}\pi^{-n}(\RE z)^{2k}\hspace{-4pt}\int_{\Complexn}\hspace{-4pt}e^{-|z-w|^{2}}d\mu(w),\quad z\in\Complexn.
	\end{align}	
		\begin{theorem}[\textbf{Criterion for horizontal $k$-FC measures}]\label{thm-char-2kfc}
			Let  $k\in(\Entero_{+}/2)^{n}$  and $\mu$ be a positive $k$-FC measure for $\Fock$. Then the following conditions are equivalent:
			\begin{enumerate}
				\item $\bT_{\ppartial_{R}^{2k}\mu}$ is horizontal.
				\item  $\widetilde{\ppartial_{R}^{2k}\mu}$ depends only on $\RE \,z$.
				\item  $\mu$ is invariant under horizontal translations, i.e., for every Borel set $X\subset \Complexn$ and every $h\in\Realn$, $\mu(X+ih)=\mu(X).$
				\item For every Borel sets $Y,Z\subset\Realn$, and every $h\in\Realn\,$,
				$\mu(Y\times (Z+h))=\mu(Y\times Z).$
				\item $\mu$ is a horizontal measure, i.e., there exists  $\varrho\in\Borelprn$  such that $\mu=\varrho\otimes\nu_{n}.$
			\end{enumerate}
		\end{theorem}
		\begin{proof}
	The proof is almost the same as the one of  Theorem \ref{thm-char-hfc},  here we  use Lemma \ref{lemma-medida-zkzero} for any $|k|\neq0$.
		\end{proof}
	By Theorem \ref {thm-char-2kfc}, every Toeplitz operator $\bT_{\ppartial_{R}^{2k}\mu}$ with real coderivative $\ppartial_{R}^{2k}\mu$ of order $2k$ of some horizontal measure $\mu$ as symbol  is a horizontal operator. Hence,  by Proposition \ref{criterio-horizontal-operators}, it is unitarily equivalent to the multiplication operator by certain $L_{\infty}$-function.  We give now the explicit formula for this $L_{\infty}$-function.
		\begin{theorem}[\textbf{Diagonalization of $\bT_{\ppartial_{R}^{2k}\mu}$}]\label{diag-To-khFC}
			Let  $k\in(\Entero_{+}/2)^{n}$ and $\mu= \varrho\otimes\nu_{n}$ be a $k$-hFC measure for $\Fock$. Then the Toeplitz operator $\bT_{\ppartial_{R}^{2k}\mu}$ is unitarily equivalent to $\Bb \bT_{\ppartial_{R}^{2k}\mu}\Bb^{*}=\gamma_{\varrho,2k}\mathrm{Id},$
			where the function $\gamma_{\varrho,2k}\colon\Realn\rightarrow\Complex$ is given by
			\begin{equation}\label{gamma-varro-x}
			\gamma_{\varrho,2k}(x)=
			 \left(\frac{2}{\pi}\right)^{n/2}\int_{\Realn}\h_{2k}^{(n)}(\sqrt{2}x-y)e^{-(x-\sqrt{2}y)^{2}}d\varrho(y),\quad x\in\Realn ;
			\end{equation}
			here $\h_{m}^{(n)}$ is the product of the Hermite polynomials $\h_{m_{j}},$ i.e.,
			\begin{equation*}
			 \h_{m}^{(n)}(t)=\prod_{j=1}^{n}(-1)^{m_{j}}e^{t_{j}^{2}}\frac{d^{m_{j}}}{dt_{j}^{m_{j}}}(e^{-t_{j}^{2}}),\quad t=(t_{j})_{j=1}^{n}\in\Realn.
			\end{equation*}
		\end{theorem}
		\begin{proof}
			Let $S=\Bb^{*}M_{\gamma_{\varrho,2k}}\Bb$. Then, by \eqref{Bargmann-inversa-kernel} and Tonelli's theorem, a simple computation shows that for every $z=u+iv\in\Complex$
			\begin{align*}
			 \widetilde{S}(u+iv)&=e^{-|u+iv|^{2}}\int_{\Realn}\gamma_{\varrho,2k}(x)|(\Bb K_{u+iv})(x)|^{2}dx\\
			 &=\frac{2^{n/2}}{\pi^{n}}\int_{\Realn}\left(\int_{\Realn}\h_{2k}^{(n)}(\sqrt{2}x-y)e^{-(x-\sqrt{2}y)^{2}}e^{-(\sqrt{2}u-x)^{2}}dx\right)d\varrho(y)\\
			 &=\frac{2^{n/2}}{\pi^{n}}\int_{\Realn}e^{-(y-u)^{2}}\left(\int_{\Realn}\h_{2k}^{(n)}(\sqrt{2}x-y)e^{-(\sqrt{2}x-(y+u))^{2}}dx\right)d\varrho(y)\\
			 &=\left(\frac{1}{\pi^{n}}\int_{\Realn}e^{-(y-u)^{2}}d\varrho(y)\right)\left(\int_{\Realn}\h_{2k}^{(n)}(t)e^{-(t-u)^{2}}dt\right)\\
			 &=\left(\frac{1}{\pi^{n}}\int_{\Realn}e^{-(y-u)^{2}}d\varrho(y)\right)\left(\prod_{j=1}^{n}\int_{\Realn}\h_{2k_{j}}(t_{j})e^{-(t_{j}-u_{j})^{2}}dt_{j}\right)\\				 &=\frac{2^{2|k|}(\RE \,z)^{2k}}{\pi^{n/2}}\int_{\Realn}e^{-(y-u)^{2}}d\varrho(y)
			\end{align*}
 (\cite[Formula 7.374- 6]{Gradshteyn} applied $n$-times.)			By \eqref{Bere-real-code-mu} and  \eqref{berezin-convo}, the last integral is exactly the Berezin transform   $\widetilde{ \ppartial_{R}^{2k}\mu}$ of $\bT_{\ppartial_{R}^{2k}\mu}$,  where  $\mu=\varrho\otimes\nu_{n}$. Therefore, by the injectivity property of the Berezin transform,  $S=\bT_{\ppartial_{R}^{2k}\mu}$, with $\mu= \varrho\otimes \nu_{n}$.
		\end{proof}

	\section{  $(\alpha,k)$-hFC type measures for $\Fock$  with $k\geq \alpha$}\label{Section-Hor-akFC}
	Note that if $\mu=\varrho\otimes\nu_{n}$ is a $k$-FC type measure for $\Fock$, then $\mu_{k}$, given by \eqref{measure-muk}, has the form
	$\mu_{k}=\varrho_{k}\otimes \nu_{n,-k}$, where
	\begin{align*}
	\varrho_{k}(X)&=\int_{X}\prod_{j=1}^{n}(1+x_{j}^{2})^{k_{j}}d\varrho(x),\quad X\in\Borelrn,\\
	 \nu_{n,-k}(Y)&=\int_{Y}\prod_{j=1}^{n}\frac{dy_{j}}{(1+y_{j}^{2})^{-k_{j}}},\quad Y\in\Borelrn.
	\end{align*}
	So, if $\mu$ is a $k$-hFC then $\mu_{k}$ is not a horizontal FC type measure for $\Fock$. To handle this circumstance, we introduce the notion of   $\alpha$-horizontal objects.
	\begin{definition}[\textbf{$\alpha$-Horizontal measures}]\label{ahor-meas}
		Let $\alpha\in\Entero^{n}$. We say that $\mu\in\Borelpcn$ is $\alpha$-\emph{horizontal} if there exists $\varrho\in\Borelprn$ satisfying $\mu=\varrho\otimes\nu_{n,\alpha}$, where
		\begin{equation*}
		 d\nu_{n,\alpha}(y)=\prod_{j=1}^{n}\frac{dy_{j}}{(1+y_{j}^{2})^{\alpha_{j}}}.
		\end{equation*}
	 Furthermore, if  $\mu=\varrho\otimes\nu_{n,\alpha}$ is a $k$-FC type measure for $\Fock$ we say that $\mu$ is $(\alpha,k)$\emph{-hFC type measure}; if $|\alpha|=|k|=0$ the measure $\mu$ is  an \emph{hFC}, as defined previously.
	\end{definition}

	\begin{proposition} \label{a-k-hFC--a-2a-k-hFC}
		Let $\alpha\in\Entero^{n}$ and  $k,p\in(\Entero_{+}/2)^{n}$. A  complex  Borel measure $\mu$  is  an $(\alpha,k)$-hFC for $\Fock$ if and only if $\mu_{k-p}$ is a $(p+\alpha-k,p)$-hFC.
	\end{proposition}
\begin{proof}
	Suppose that $\mu$ is a positive $(\alpha,k)$-hFC measure for $\Fock$. Then  $\mu$ is a $k$-FC measure for $\Fock$ such that $\mu=\varrho\otimes\nu_{n,\alpha}$ for some regular measure $\varrho$ of $\Realn.$ By Proposition \ref{prop-Ck},   $\mu_{k-p}$ is a $p$-FC measure for $\Fock$ and
	\begin{align*}
	\mu_{k-p}(X)&=\int_{X} \prod_{j=1}^{n}(1+x_{j}^{2})^{k_{j}-p_{j}}(1+y_{j}^{2})^{k_{j}-p_{j}}d\varrho(x)d\nu_{n,\alpha}(y)\\
	&=\varrho_{k-p}\otimes \nu_{n,p+\alpha-k}(X),\quad \text{for any Borel set $X$ of $\Complexn$.}
	\end{align*}
	Here $\varrho_{k-p}$ denotes  the positive measure
	 $$\varrho_{k-p}(Y)=\int_{Y}\prod_{j=1}^{n}(1+x_{j}^{2})^{k_{j}-p_{j}}d\varrho(x).$$
	
	\noindent
	Conversely, if $\mu_{k-p}$ is a $(p+\alpha-k,p)$-hFC then $\mu_{k-p}$ is a $p$-FC type measure for $\Fock$ and hence $\mu$ is $k$-FC type measure for $\Fock$, by Proposition \ref{prop-Ck}. Moreover, there exists a $\varrho\in\Borelprn$ such that
	$\mu_{k-p}=\varrho\otimes\nu_{n,p+\alpha-k}$. Thus,
	\begin{align*}
	\mu(X)&=\int_{X} \prod_{j=1}^{n}(1+x_{j}^{2})^{p_{j}-k_{j}}(1+y_{j}^{2})^{p_{j}-k_{j}}d\mu_{k-p}(y)\\
	&=\int_{A} \prod_{j=1}^{n}(1+x_{j}^{2})^{p_{j}-k_{j}}(1+y_{j}^{2})^{p_{j}-k_{j}}d\varrho(x)d\nu_{n,p+\alpha-k}(y)\\
	&=\varrho_{p-k}\otimes \nu_{n,\alpha}(X),\quad \text{for any Borel set $X\subset \Complexn$.}
	\end{align*}
Here $\varrho_{p-k}$ denotes  the positive measure
$\displaystyle\varrho_{p-k}(Y)=\int_{Y}\prod_{j=1}^{n}(1+x_{j}^{2})^{p_{j}-k_{j}}d\varrho(x).$
\end{proof}
\noindent	
Now, if  $\alpha,k\in(\Entero_+/2)^{n}$, with $k\geq \alpha$ and  $\mu=\varrho\otimes\nu_{n,\alpha}$ is  a positive  $(\alpha,k)$-hFC type measure for $\Fock$, then, by Proposition \ref{a-k-hFC--a-2a-k-hFC},  $\mu_{\alpha}$ is a $(0,k-\alpha)$-hFC type measure for $\Fock$. i.e., $\mu_{\alpha}=\varrho_{\alpha}\otimes\nu_{n,0}$, with   $ d\varrho_{\alpha}(y)=\prod_{j=1}^{n}(1+y_{j}^{2})^{\alpha_{j}}d\varrho(y)$, is a $(k-\alpha)$-FC measure, which is horizontal. Hence, by Theorem \ref{thm-char-2kfc}, the Toeplitz operator $\bT_{\ppartial_{R}^{2(k-\alpha)}\mu_{\alpha}}$ is horizontal and, by Theorem \ref{diag-To-khFC}, it is unitarily equivalent to $\Bb \bT_{\ppartial_{R}^{2(k-\alpha)}\mu_{\alpha}}\Bb^{*}=M_{\gamma_{\varrho_{\alpha},2(k-\alpha)}}$, where  $\gamma_{\varrho_{\alpha},2(k-\alpha)}$ is as in \eqref{gamma-varro-x}.

	\section{$\Ln$-invariant $k$-FC type measures}\label{Section-LFC}
	\noindent
	Let $\Ln$ be any Lagrangian plane of the symplectic vector space $(\Real^{2n},\omega_{0})$. In this section  we extend the results on $k$-hFC type measures to  $\Ln$-invariant $k$-FC type measures.
	
	\begin{definition}[\textbf{$\Ln$-FC}]\label{Linv-meas}
		Let  $\Ln\in\operatorname{Lag}(2n,\Real)$. A complex valued  Borel regular measure $\mu$ on $\Complexn$ is said to be  invariant under \emph{Lagrangian translations} ($\Ln$-invariant, in short) if for each Borel subset $E$ of $\Complexn$ and for  every $h\in\Ln$ it satisfies
		$$\mu(E-h)=\mu(E).$$
		In particular,   $\Ln=i\Realn$ corresponds to the horizontal case. Furthermore,  if  $k\in(\Entero_+/2)^n$  and  $\mu$ is a $k$-FC type measure for $\Fock$, which is  $\Ln$-invariant, then  we say simply that $\mu$ is  $k$-$\Ln$-FC.
	\end{definition}
	\noindent
	By the transitivity property of $\operatorname{U}(2n,\Real)$ \cite[Proposition 43]{Gosson} and due to the isomorphism $\operatorname{U}(2n,\Real)\simeq\operatorname{U}(n,\Complex)$,  there exists a unitary matrix $\mathbf{X}\in\operatorname{U}(n,\Complex)$ such that $\mathbf{X}\Ln=i\Realn.$
For  $\mathbf{X}\in\operatorname{U}(n,\Complex)$,   we denote by $V_{\mathbf{X}}$ the linear operator  $V_{\mathbf{X}}\colon\Ele(\Complexn,d\g_{n})\rightarrow\Ele(\Complexn,d\g_{n})$ given by
	\begin{equation}\label{eq-VB}
	(V_{\mathbf{X}}f)(z)=f(\mathbf{X}^{*}z),\quad z\in\Complexn.
	\end{equation}
	\noindent
	Since  $\mathbf{X}^{*}=\mathbf{X}^{-1}\in\operatorname{U}(n,\Complex)$,   $V_{\mathbf{X}}$ is a unitary operator, with $V_{\mathbf{X}}^{*}=V_{\mathbf{X}^{-1}}$.
 A Borel measure $\mu$ on $\Complexn$ pushes forward to the measure  $\mu_{\mathbf{X}}$ on $\Complexn$ by
	\begin{equation}\label{med-muX}
	\mu_{\mathbf{X}}(E)=\mu(\mathbf{X}E)=\mu\left(\left\{\mathbf{X}z\colon z\in\,E\right\}\right), \quad \text{for all Borel sets }\,E.
	\end{equation}
	Observe that the regularity of $\mu$ implies the regularity of $\mu_{\mathbf{X}}$. Furthermore,   $\mu$ is a FC measure for $\Fock$ provided  $\mu_{\mathbf{X}}$ is FC; then, by \eqref{eq-VB}, for every $f\in\Fock$,
	\begin{align*}	 \int_{\Complexn}|f(w)|^{2}d|\mu_{\mathbf{X}}|(w)&=\int_{\Complexn}|f(\mathbf{X}z)|^{2}d|\mu|(z)=\int_{\Complexn}|(V_{\mathbf{X}^{-1}}f)(z)|^{2}d|\mu|(z).
	\end{align*}
	In addition,  $\mu$ is a hFC  type measure iff $\mu_{\mathbf{X}}$ is $\Ln$-FC, since for any Borel set $E$
	 $$\mu_{\mathbf{X}}(E-h)=\mu(\mathbf{X}E-\mathbf{X}h)=\mu(\mathbf{X}\mathbf{}E-it)=\mu(\mathbf{X}E)=\mu_{\mathbf{X}}(E).$$
	Therefore, if $\mu$ is a $\Ln$-FC measure, then $\mu_{\mathbf{X}^{*}}$ is horizontal and, by Theorem \ref{thm-char-hfc}, $\mu_{\mathbf{X}^{*}}=\varrho\otimes\nu_{n}$ for some Borel regular $\varrho$. Since the matrix  $\mathbf{X}\in\operatorname{U}(n,\Complex)$ satisfying $\mathbf{X}\Ln=i\Realn$ is not unique, therefore, for any other $\mathbf{Y}\in\operatorname{U}(n,\Complex)$ such that $\mathbf{Y}\Ln=i\Realn$, the relation $\mu_{\mathbf{Y}^{*}}=\eta\otimes\nu_{n}$ holds and
	 $$\mu_{\mathbf{X}^{*}}(E)=\mu_{\mathbf{Y}^{*}}(\mathbf{Y}\mathbf{X}^{*}E),\quad E\in\Borelcn.$$
	In addition, it is easy to check that $t\mapsto \mathbf{Y}\mathbf{X}^{*}t$ is an automorphism of $i\Realn$.
	\begin{proposition}
		Let  $\Ln\in\operatorname{Lag}(2n,\Real)$,   $\mathbf{X}\in\operatorname{U}(n,\Complex)$ be fixed  such that $\mathbf{X}\Ln=i\Realn$. If  $k\in(\Entero_{+}/2)^{n}$  then the C*-algebra $\mathcal{T}(k\text{-}\Ln\text{-}FC)$ generated by  Toeplitz operators $\bT_{\ppartial_{R}^{2k}\mu}$  is isometrically isomorphic to the C*-algebra $\mathcal{T}(k\text{-}hFC)$ generated by Toeplitz operators $\bT_{\ppartial_{R}^{2k}\eta}$.
	\end{proposition}
	\begin{example}
		Let $\Ln=\Realn\times\{0\}$ and $\mu$ be an $(\alpha,k)$-$\Ln$-FC. Observe that the standard symplectic matrix $\mathbf{J}$  rotates  the Lagrangian plane $\{0\}\times \Realn$ to $\Realn\times\{0\}$. i.e., $$\left(\begin{matrix}0&I_{n}\\-I_{n}&0\end{matrix}\right)\left(\begin{matrix}0\\x\end{matrix}\right)=\left(\begin{matrix}x\\0\end{matrix}\right)$$
		Therefore,  if $\mathbf{X}=-iI_{n}\in\operatorname{U}(n,\Complex)$, then by \eqref{med-muX},
		\begin{align*}
		\mu(E\times F)&=\mu(\mathbf{X}\mathbf{X}^{*}(E\times F))=\mu_{\mathbf{X}}(\mathbf{X}^{*}(E\times F))=\mu_{\mathbf{X}}((-F)\times E)=\varrho(-F)\nu_{n}(-E).
		\end{align*}
		Thus, by Proposition \ref{a-k-hFC--a-2a-k-hFC}, the measure  $\mu_{\alpha}$ is a $(0,k-\alpha)$-hFC for $\Fock$, where $\mu_{\alpha}=\varrho_{\alpha}\otimes\nu_{n,0}$ and  $\displaystyle d\varrho_{\alpha}(y)=\prod_{j=1}^{n}(1+y_{j}^{2})^{\alpha_{j}}d\varrho(y)$. Therefore, by \eqref{gamma-varro-x}, the corresponding spectral function is given by
			\begin{equation*}
		\gamma_{\varrho_{\alpha},2(k-\alpha)}(x)=
		 \left(\frac{2}{\pi}\right)^{n/2}\int_{\Realn}\h_{2(k-\alpha)}^{(n)}(\sqrt{2}x-y)e^{-(x-\sqrt{2}y)^{2}}\prod_{j=1}^{n}(1+y_{j}^{2})^{\alpha_{j}}d\varrho(y),\quad x\in\Realn;
		\end{equation*}
		here $\h_{2(k-\alpha)}^{(n)}$ is the product of the Hermite polynomials $\h_{2(k_{j}-\alpha_{j})}$. i.e.,
		\begin{equation*}
		 \h_{2(k-\alpha)}^{(n)}(t)=\prod_{j=1}^{n}e^{t_{j}^{2}}\frac{d^{2(k_{j}-\alpha_{j})}}{dt_{j}^{2(k_{j}-\alpha_{j})}}(e^{-t_{j}^{2}}),\quad t=(t_{j})_{j=1}^{n}\in\Realn.
		\end{equation*}
		\noindent
		On the other hand, let $\Delta=\left\{(x,x)\in\Real^{2n}\colon x\in\Realn\right\}$ and $\mu$ be a $k$-$\Delta$FC. Then,  since $$\left(\begin{matrix}I_{n}&I_{n}\\-I_{n}&I_{n}\end{matrix}\right)\left(\begin{matrix}0\\x\end{matrix}\right)=\left(\begin{matrix}x\\x\end{matrix}\right)$$
		 if $\mathbf{X}=\frac{1}{2}\left(I_{n}-iI_{n}\right)\in\operatorname{U}(n,\Complex)$,  and  by \eqref{med-muX},
		\begin{align*}
		\mu\left(E \times F\right)&=\mu\left(\mathbf{X}\mathbf{X}^{*}\left(E\times F\right)\right)=\mu_{\mathbf{X}}(\mathbf{X}^{*}(E\times F))=\mu_{\mathbf{X}}\left(\frac{(E+F)}{2}\times \frac{(F-E)}{2}\right)\\
		&=\varrho\left(\frac{E+F}{2}\right)\nu_{n}\left(\frac{F-E}{2}\right).
		\end{align*}
	\end{example}
\begin{definition}[\textbf{$\alpha\Ln$-invariant}]\label{aLinv-meas}
	Let  $\Ln\in\operatorname{Lag}(2n,\Real)$,   $\mathbf{X}\in\operatorname{U}(n,\Complex)$ be fixed and such that $\mathbf{X}\Ln=i\Realn$. 	If $\alpha\in\Entero^{n}$ then we say that $\mu\in\Borelpcn$ is $\alpha\Ln$-\emph{invariant} if $\mu_{\mathbf{X}^{*}}$ is $\alpha$-horizontal. Furthermore,  if  $k\in(\Entero_{+}/2)^{n}$ and  $\mu$ is a $k$-FC type measure for $\Fock$, which is  $\alpha\Ln$-invariant, then  we say simply that $\mu$ is  $(\alpha,k)$-$\Ln$-FC.
\end{definition}

		\begin{theorem}\label{thm-7-4}
		Let  $\Ln\in\operatorname{Lag}(2n,\Real)$,   $\mathbf{X}\in\operatorname{U}(n,\Complex)$ be fixed  and $\mathbf{X}\Ln=i\Realn$. If  $\alpha\in\Entero^{n}$ and $k\in (\mathbb{Z}_{+}/2)^{n}$, satisfying $k\geq \alpha$, then the C*-algebra $\mathcal{T}((\alpha,k)\text{-}\Ln\text{-}FC)$ generated by the Toeplitz operators $\bT_{\ppartial_{R}^{2(k-\alpha)}\mu_{\mathbf{X}^{*},\alpha}}$  is isometrically isomorphic to $\mathcal{T}((\alpha,k)\text{-}hFC)$.
	\end{theorem}

Finally, we express the conditions of Theorems \ref{thm-char-2kfc},  \ref{diag-To-khFC} and \ref{thm-7-4} in terms of the measure $\mu$ itself. To do this, we introduce $\Ln$-real coderivatives.

	\begin{definition}[\textbf{$\Ln$-real coderivatives}]
		Let  $\Ln\in\operatorname{Lag}(2n,\Real)$,   $\mathbf{X}\in\operatorname{U}(n,\Complex)$ be fixed  such that $\mathbf{X}\Ln=i\Realn$. If  $k\in(\Entero_{+}/2)^{n}$ and  $\mu$ is a $k$-FC type measure for $\Fock$,  then
		\begin{equation}
		\ppartial_{R,\Ln}^{2k}\mu=\ppartial_{R}^{2k}\mu_{\mathbf{X}^{*}}
		\end{equation}
		is called the $k$th \emph{$\Ln$-real coderivative} of the measure $\mu$.
	\end{definition}
\begin{theorem}[\textbf{Criterion for  $k$-$\Ln$-FC measures}]\label{thm-char-2klfc}
		Let  $\Ln\in\operatorname{Lag}(2n,\Real)$,   $\mathbf{X}\in\operatorname{U}(n,\Complex)$ be fixed  such that $\mathbf{X}\Ln=i\Realn$,  $k\in(\Entero_{+}/2)^{n}$  and $\mu$ be a complex $k$-FC measure for $\Fock$. Then the following conditions are equivalent:
	\begin{enumerate}
		\item $\W_{-h}\bT_{\ppartial_{R,\Ln}^{2k}\mu}\W_{h}=\bT_{\ppartial_{R,\Ln}^{2k}\mu}$ for every $h\in\Ln$.
		\item  $\widetilde{\ppartial_{R,\Ln}^{2k}\mu}$ is a $\Ln$-invariant function.
		\item  $\mu$ is $\Ln$-invariant, i.e., for every Borel set $X\subset \Complexn$ and every $h\in\Ln$, $\mu(X+h)=\mu(X).$
		\item $\mu_{\mathbf{X}^{*}}$ is a horizontal measure, i.e., there exists  $\varrho\in\Borelprn$  such that $\mu_{\mathbf{X}^{*}}=\varrho\otimes\nu_{n}.$
	\end{enumerate}
\end{theorem}
	Let  $\Ln\in\operatorname{Lag}(2n,\Real)$,   $\mathbf{X}\in\operatorname{U}(n,\Complex)$ be fixed  such that $\mathbf{X}\Ln=i\Realn$,  $k\in(\Entero_{+}/2)^{n}$  and $\mu$ be a complex $k$-$\Ln$-FC measure for $\Fock$. Then, $\mu_{\mathbf{X}^{*}}$ is horizontal. i.e.,  $\mu_{\mathbf{X}^{*}}=\varrho\otimes\nu_{n}$,  and if we write $\mathbf{X}$ as  $$\mathbf{X}=\mathbf{X}_{1}+\mathbf{X}_{2},\,\text{where}\, \mathbf{X}_{1}=\left(\begin{matrix}
	A&B\\0&0
	\end{matrix}\right),\,\mathbf{X}_{2}=\left(\begin{matrix}
	0&0\\C&D
	\end{matrix}\right)$$ we obtain that $\mu(Y\times Z)=\mu_{\mathbf{X}^{*}}(\mathbf{X}(Y\times Z))=\varrho(\mathbf{X}_{1}(Y\times Z))\nu_{n}(\mathbf{X}_{2}(Y\times Z))$ for every $Y,Z\in\Borelrn$. This implies that $\mu(Y\times Z)=\varrho_{\mathbf{X}}\otimes \nu_{n,\mathbf{X}}(Y\times Z)=\varrho(\mathbf{X}_{1}(Y\times Z))\nu_{n}(\mathbf{X}_{2}(Y\times Z))$ for every $Y,Z\in\Borelrn$.
\begin{theorem}[\textbf{Diagonalization of $\bT_{\ppartial_{R,\Ln}^{2k}\mu}$}]\label{diag-To-klFC}
	Let  $\Ln\in\operatorname{Lag}(2n,\Real)$,   $\mathbf{X}\in\operatorname{U}(n,\Complex)$ be fixed  such that $\mathbf{X}\Ln=i\Realn$,  $k\in(\Entero_{+}/2)^{n}$ and $\mu$ be a $k$-$\Ln$-FC measure for $\Fock$. Then the Toeplitz operator $\bT_{\ppartial_{R}^{2k}\mu}$ is unitarily equivalent to $\Bb \bT_{\ppartial_{R,\Ln}^{2k}\mu}\Bb^{*}=\gamma_{\varrho,2k}\mathrm{Id},$
	where $\mu_{\mathbf{X}^{*}}=\varrho\otimes\nu_{n}$ and $\gamma_{\varrho,2k}$ is given in \eqref{gamma-varro-x}.
\end{theorem}
	\begin{theorem}\label{thm-7-4-vLrd}
	Let  $\Ln\in\operatorname{Lag}(2n,\Real)$,   $\mathbf{X}\in\operatorname{U}(n,\Complex)$ be fixed  and $\mathbf{X}\Ln=i\Realn$. If  $\alpha\in\Entero^{n}$ and $k\in (\mathbb{Z}_{+}/2)^{n}$, satisfying $k\geq \alpha$, then the C*-algebra $\mathcal{T}((\alpha,k)\text{-}\Ln\text{-}FC)$ generated by the Toeplitz operators $\bT_{\ppartial_{R,\Ln}^{2(k-\alpha)}\mu_{\alpha}}$  is isometrically isomorphic to $\mathcal{T}((\alpha,k)\text{-}hFC)$.
\end{theorem}

\section{Common properties}\label{se:common}	

We denote by $\mathcal{T}(hFC)$, $\mathcal{T}(k\text{-}hFC)$, $\mathcal{T}((\alpha,k)\text{-}hFC)$, and $\mathcal{T}(k\text{-}\mathcal{L}\text{-}FC)$ the $C^*$-algebras generated by all Toeplitz operators defined by corresponding measure-symbols: horizontal FC, horizontal $k$-FC, horizontal $(\alpha,k)$-FC, and $\mathcal{L}$-invariant $k$-FC (defined in Sections \ref{Section-Hor-FC}, \ref{Section-Hor-kFC},
\ref{Section-Hor-akFC}, and \ref{Section-LFC}), respectively.

The main result in these sections establishes a unitary equivalence of Toeplitz operators $\bT_{*}$ in question with multiplication operators by their spectral functions $\gamma_{*}$, and gives the precise formulas for those functions, specific for each type of measure-symbols under consideration.

We denote as well by $\mathfrak{G}(hFC)$, $\mathfrak{G}(k\text{-}hFC)$, $\mathfrak{G}((\alpha,k)\text{-}hFC)$, and $\mathfrak{G}(k\text{-}\mathcal{L}\text{-}FC)$ the $C^*$-algebras generated  by all spectral functions $\gamma_{*}$, specific for each of the above four types of measure-symbols.

Let finally $\mathcal{T}(*)$ be either one of the above Toeplitz operator algebras and $\mathfrak{G}(*)$ be the corresponding function algebra. Then, as it was mentioned in Introduction, the diagonalization results obtained reveal the majority of the main properties of the corresponding Toeplitz operators.

\begin{theorem}
 The C*-algebra $\mathcal{T}(*)$ is commutative and  isometrically isomorphic (and even unitarily equivalent) to the C*-algebra $\mathfrak{G}(*)$.
The isomorphism
\begin{equation*}
 \pi_{*} : \ \mathcal{T}(*) \ \ \longrightarrow \ \ \mathfrak{G}(*)
 \end{equation*}
is generated by the following mapping
\begin{equation*}
 \pi_{*} : \ \bT_{*} \ \ \longmapsto \ \ \gamma_{*}.
\end{equation*}
The norm of each operator $\bT \in \mathcal{T}(*)$ coincides  with its spectral radius and is given by
\begin{equation*}
 \|\bT\| = r_{\bT} = \sup |\pi_{*}(\bT)|.
\end{equation*}
The spectrum of each operator $\bT \in \mathcal{T}(*)$ coincides with the closure of the range of its spectral function $\gamma = \pi_{*}(\bT)$.\\
Each common invariant subspace of $\Fock$ for all operators from $\mathcal{T}(*)$ has the form $\Bb^{*}(L_2(M))$, where $M$ is a measurable subset of subset of $\mathbb{R}^n$.
\end{theorem}

	

\end{document}